\documentclass{amsart}
\usepackage{amssymb}
\usepackage{amsmath}
\usepackage{textcomp}
\makeindex

\newcommand{\ba}{\begin{array}}
\newcommand{\eea}{\end{eqnarray}}
\newcommand{\ea}{\end{array}}

\newtheorem{definition}{Definition}[section]
\newtheorem{theorem}[definition]{Theorem}
\newtheorem{lemma}[definition]{Lemma}
\newtheorem{proposition}[definition]{Proposition}

\newtheorem{remark}[definition]{Remark}

\usepackage[matrix,arrow,curve]{xy}
\usepackage[utf8x]{inputenc}
\usepackage{tikz}

\begin{document}
\title[Characterizing Regular Lagrangians by Lefschetz Fibration]{Characterizing Regular Lagrangians by Lefschetz Fibration}
\author[S. Mukherjee]{Sauvik Mukherjee}
\address{mukherjeesauvik@gmail.com\\}

\begin{abstract} 
In \cite{Eliashberg} Eliashberg, Ganatra and Lazarev have introduced the notion of \textit{Regular Lagrangians} in \textit{Weinstein cobordisms} and also predicted that \textit{Regular Lagrangians} can be characterized by the existence of \textit{Weinstein Lefschetz Fibrations}. In this paper we present a proof of this fact with an hypothesis.
 \end{abstract}
\maketitle

\section{introduction} In \cite{Eliashberg} Eliashberg, Ganatra and Lazarev have introduced the notion of \textit{Regular Lagrangians} in \textit{Weinstein cobordisms} and also predicted that \textit{Regular Lagrangians} can be characterized by the existence of \textit{Weinstein Lefschetz Fibrations}. In this paper we present a proof of this fact with an hypothesis.\\

Let us first define the notion of \textit{Regular Lagrangians}. But in order to do so we need to understand the notion of \textit{Liouville cobordisms}.\\

A \textit{Liouville cobordism} between contact manifolds $(\partial_{\pm}W,\xi_{\pm})$ is an even dimensional cobordism (say of dimension $2n$) $(W,\partial_{-}W,\partial_{+}W)$, where $W$ is equipped with a symplectic form $\omega$ and an expanding Liouville vector field $X$ which is outward pointing along $\partial_{+}W$ and inward pointing along $\partial_{-}W$, moreover the contact structure induced by the Liouville form $\lambda=\iota(X)\omega$ on $\partial_{\pm}W$ coincides with the given ones.\\

A \textit{Liouville cobordism} $(W,\partial_{-}W,\partial_{+}W)$ is called a \textit{Weinstein cobordism} if there exists a defining Morse function $\phi:W\to \mathbb{R}$ for which $X$ is gradient-like and $(\omega,X,\phi)$ is called a \textit{Weinstein cobordism} structure on $W$. Equivalently if the Liouville form $\lambda$ is given then $\omega$ and $X$ can be recovered, so we may say $(\lambda, \phi)$ to be a \textit{Weinstein structure}. A \textit{Weinstein cobordism} $(W,\partial_{-}W,\partial_{+}W)$ with $\partial_{-}W=\Phi(empty)$ is called a \textit{Weinstein domain}.\\

In what follows we shall consider \textit{exact Lagrangian subcobordism} with Legendrian boundary $(L,\partial_{-}L,\partial_{+}L)\subset (W,\partial_{-}W,\partial_{+}W)$ in \textit{Weinstein cobordism} with contact boundary.\\

Now consider $(L,\partial_{-}L,\partial_{+}L)\subset (W,\partial_{-}W,\partial_{+}W)$ an \textit{exact Lagrangian cobordism} in a \textit{Liouville cobordism} and let $N(L)$ be a small tubular neighborhood of $L\subset W$. Define $\partial^{int}N(L)=\partial N(L)\cap Int(W)$, $\partial^{ext}_{\pm}N(L)=\partial N(L)\cap \partial_{\pm}W$ and $W_{L}=\overline{W-N(L)}$. As $L$ is exact we can make $X\pitchfork \partial N(L)$ and hence $(W_{L},\omega_{\mid W_{L}})$ has a natural \textit{Liouville cobordism} structure with sutured boundary $\partial_{+}W_{L}\cup \partial_{-}W_{L}$, where \[\partial_{+}W_{L}=\partial W-\partial^{ext}N(L),\ \partial_{-}W_{L}=\partial^{int}N(L)\] 

\begin{definition}(\cite{Eliashberg})
\label{Regular Lagrangian}
Let $L\subset (W,\omega',X',\phi')$ be an exact Lagrangian cobordism in a Weinstein cobordism. $L$ is called regular if the Weinstein structure is homotopic to a Weinstein cobordism structure $(W,\omega,X,\phi)$ through a homotopy of Weinstein structures for which $L$ remains Lagrangian and $X$ is tangent to $L$. This is equivalent to the condition that the Liouville form $\lambda:=\iota(X)\omega$ has the property that $\lambda_{\mid L}=0$.
\end{definition}

Such a Weinstein structure is called \textit{tangent Weinstein structure} to the regular Lagrangian $L$. It follows that the critical points of $\phi_{\mid L}$ are global critical points for $W$.\\

It turns out that for a regular Lagrangian $L$ and any Weinstein structure tangent to $L$ can further be adjusted so that the following happens.\\

The Weinstein cobordism structure $(W,\omega,X,\phi)$ tangent to $L$ is called adjusted to $L$ if there exists a regular value $c\in \mathbb{R}$ of $\phi$ such that 

\begin{enumerate}
\item all critical points of $\phi$ in the sublevel set $\{\phi\leq c\}$ lies on $L$ and the indeces of these critical points for $\phi$ and $\phi_{\mid L}$ coincides;\\
\item there are no critical points of $\phi$ on $L\cap \{\phi\geq c\}$.
\end{enumerate}

So the handlebody presentation has the following description. First the handles corresponding to the critical points of $\phi_{\mid L}$ are attached and then the remaining ones are attached so that the attaching spheres of these remaining handle attachments do not intersect $\partial_{+}L$.\\

It turns out that if $\partial_{-}L$ is empty then the Weinstein structure on $W(L)=\{\phi\leq c\}$ is equivalent(radially) to the disjoint union of the trivial cobordism over $\partial_{-}W$, i.e,  $\partial_{-}W\times [0,1]$ and the canonical Weinstein structure on $T^*L$.\\

If $\partial_{-}L$ is non-empty then the handlebody presentation of $L$ given by $\phi_{\mid L}$ and the gradient-like vector field $X_{\mid L}\in TL$ defines a weinstein cobordism structure on the cotangent disk bundle $N$ of $L$ with $\partial_{-}N$ equal to a tubular neighborhood of $\partial_{-}L$ in $J^1(\partial_{-}L)$ and $\partial_{+}N=\partial N- Int\partial_{-}N$. Then $W(L)$ is obtained by attaching a generalized handle, i.e, by attaching the Weinstein cobordism $N$ to the trivial cobordism $\partial_{-}W\times [0,1]$ over $\partial_{-}W$ by identifying $\partial_{-}L$ with a neighborhood of a Legendrian sphere in $\partial_{-}W\times \{1\}$.\\

So in any case the underlying symplectic structure on $W(L)$ near $L$ is given by the symplectic structure on $T^*L$ but the Weinstein structure differs depending on whether $\partial_{-}L$ is empty or not.\\

 Next we need to define the notion of \textit{Weinstein Lefschetz fibration}.

\begin{definition}(\cite{Pardon})
\label{Weinstein Lefschetz fibration}
A Weinstein Lefschetz fibration is a tuple \[W^{2n}=((W_0,\lambda_0,\phi_0);L_1,...,L_m)\] where $W_0^{2n-2}$ is a Weinstein domain and $L_j$'s are exact parametrized Lagrangian spheres in $W_0^{2n-2}$, where parametrized means diffeomorphisms $S^{n-1}\to L$ upto precomposition with elements of $O(n)$. Its total space $|W|$ is given as follows. Consider the Weinstein manifold \[(W_0\times \mathbb{C},\lambda_0-J^*d(\frac{1}{2}|z|^2),\phi_0+|z|^2)\] The corresponding Liouville vector field is given by $X_{\lambda_0}+\frac{1}{2}(x\partial_x+y\partial_y)$. Take Legendrian lifts $\Lambda_j\subset (W_0\times S^1,\lambda_0+Nd\theta)$ of $L_j\subset W_0$ such that $\Lambda_j$ projects to small interval around $\frac{2\pi j}{m}\in S^1$. We take $N$ so large that these intervals remains disjoint. The embedding $S^1\to \mathbb{C}$ as circle of radius $N^{1/2}$ pulls back the Liouville form $-J^*d(\frac{1}{2}|z|^2)$ to the contact form $Nd\theta$. So $\Lambda_j$ could be thought as a Legendrian at the level set $\{|z|=N^{1/2}\}\subset W_0\times \mathbb{C}$. Then the restriction of downward Liouville flow defines a map $\Lambda_j\times \mathbb{R}_{\geq 0}\to W_0\times \mathbb{C}$ and which intersects the set $\{\phi_0+|z|^2=0\}$ along a Legendrian $\Lambda'_j$. Again $N$ is taken so large that the projection of $\{\phi_0+|z|^2\leq 0\}$ to $\mathbb{C}$ is contained inside the disc of radius $N^{1/2}$. Now the total space $|W|$ is defined as the result of attaching Weinstein handles (\cite{Weinstein}) to the Weinstein domain $\{\phi_0+|z|^2=0\}$ along the Legendrians $\Lambda'_j$.
\end{definition} 

Now we end this section with the main theorem of this paper.

\begin{theorem}
\label{Main}
Let $L\subset W^{2n}$ be a regular Lagrangian, then there exists a Weinstein Lefschetz fibration over $\mathbb{C}$ which projects $L$ to a ray in $\mathbb{R}\subset \mathbb{C}$ if the first cohomology of $W$ is trivial.
\end{theorem}

\begin{remark}
Observe that we have only stated the 'only if' part, the 'if' part is obvious.
\end{remark}

The method of the proof is the following. In section 4 we have constructed a sequence of complex line bundles $E^k$ over $W^{2n}$ along with sections $s_k$ such that $s_k$ has a non-degenerate singularity on $L\subset W$ and the zero locus of $s_k$ for large $k$ is Symplectic. The idea of the construction is similar to \cite{Donaldson} but it is different as we needed some improvisations.\\

Then we showed that the stable manifold, denoted $L'$, of the Hamiltonian vector field $X_{Im\log s}$ is $C^0$-close to $L$ and is diffeomorphic to $L$ (see \ref{Diff-C^0} below), where $s=s_k$ for large $k$. Then using results from section 3 and Weinstein neighborhood theorem we construct an iso-symplectic immersion which takes $L$ to $L'$ (see \ref{iso-sym pull} below). Then we pull back the Weinstein structure by this iso-symplectic immersion with some adjustment near $\partial W$. This is possible by \ref{iso-sym pull} below. Obviously under this composition $L$ is mapped to a ray in $\mathbb{C}$ because the argument is constant along $L'$. Moreover the iso-symplectic immersion is homotopic through iso-symplectic immersions to the identity and hence the pulled back Weinstein structure is homotopic through Weinstein structures to the given Weinstein structure.  \\

The above idea has been used already in \cite{Pardon}. Here we have followed it.\\

We must take $c$ so small that this tubular neighborhood of $L$ is first considered and then the rest of $W$, where $c\in \mathbb{R}$ is the regular value as in the definition of Regular Lagrangian.\\

\section{Preliminaries}First we need some notations. Let $z=(z_1,...,z_n)\in \mathbb{C}^n$ be a point in $\mathbb{C}^n$. Define 
\begin{enumerate}
\item $\frac{\partial}{\partial z_j}\equiv \frac{1}{2}(\frac{\partial}{\partial x_j}-i\frac{\partial}{\partial y_j})$, $j=1,...,n.$\\
\item $\frac{\partial}{\partial \bar{z}_j}\equiv \frac{1}{2}(\frac{\partial}{\partial x_j}+i\frac{\partial}{\partial y_j})$, $j=1,...,n.$\\
\item $dz_j\equiv dx_j+idy_j,\ j=1,...,n$\\
\item $d\bar{z}_j\equiv dx_j-idy_j,\ j=1,...,n$\\
\end{enumerate}

For $0\leq p,q\leq n$, the differential form \[\omega=\Sigma_{|\alpha|=p,|\beta|=q}\omega_{\alpha \beta}dz^{\alpha}\wedge d\bar{z}^{\beta}\] is said to be of \textit{type or bidegree} $(p,q)$, where $\alpha,\beta$ are multi-indices. Any differential form is the sum of the terms of the form $\omega_{\alpha \beta}dz^{\alpha}\wedge d\bar{z}^{\beta}$. For $\omega=\Sigma_{|\alpha|=p,|\beta|=q}\omega_{\alpha \beta}dz^{\alpha}\wedge d\bar{z}^{\beta}$ define 

\begin{enumerate}
\item $\partial \omega=\Sigma_{j=1}^n\Sigma_{\alpha,\beta} \frac{\partial\omega_{\alpha,\beta}}{\partial z_j}dz_j\wedge dz^{\alpha}\wedge d\bar{z}^{\beta}$\\
\item $\bar{\partial} \omega=\Sigma_{j=1}^n\Sigma_{\alpha,\beta} \frac{\partial\omega_{\alpha,\beta}}{\partial \bar{z}_j}d\bar{z}_j\wedge dz^{\alpha}\wedge d\bar{z}^{\beta}$\\
\end{enumerate}

Now we shall state a theorem from \cite{Mishachev} which we need for the proof of \ref{Main}. \\

\begin{theorem}(Weinstein Neighborhood Theorem, \cite{Mishachev})
\label{WNT}
Any isotropic and, in particular, Lagrangian immersion $L\to W$ extends to an isosymplectic immersion $U\to W$, where $U$ is a tubular neighborhood of the zero-section in the cotangent bundle $T^*L$ endowed with the canonical symplectic structure.
\end{theorem} 

\begin{theorem}(\cite{Datta})
\label{datta}
Let $(V,\omega_V)$ be an open symplectic manifold and $A\subset V$ be a polyhedron of positive codimension (a core), then there exists a homotopy of iso-symplectic immersions $g_t:V\to V$ such that $g_0=id$ and $g_1(V)\subset Op(A)$.
\end{theorem}

\section{$h$-Principle} This section does not have any new result, we just recall some facts from the theory of $h$-principle which we shall need in our proof.\\

 Let $X\to M$ be any fiber bundle and let $X^{(r)}$ be the space of $r$-jets of jerms of sections of $X\to M$ and $j^rf:M\to X^{(r)}$ be the $r$-jet extension map of the section $f:M\to X$. A section $F:M\to X^{(r)}$ is called holonomic if there exists a section $f:M\to X$ such that $F=j^rf$. In the following we use the notation $Op(A)$ to denote a small open neighborhood of $A\subset M$ which is unspecified. \\
 
 Let $\mathcal{R}$ be a subset of $X^{(r)}$. Then $\mathcal{R}$ is called a differential relation of order $r$. $\mathcal{R}$ is said to satisfy $h$-principle if any section $F:M\to \mathcal{R}\subset X^{(r)}$ can be homotopped to a holonomic section $\tilde{F}:M\to \mathcal{R}\subset X^{(r)}$ through sections whose images are contained in $\mathcal{R}$. Put differently,  if the space of sections of $X^{(r)}$ landing into $\mathcal{R}$ is denoted by $Sec \mathcal{R}$ and the space of holonomic sections of $X^{(r)}$ landing into $\mathcal{R}$ is denoted by $Hol \mathcal{R}$ then $\mathcal{R}$ satisfies $h$-principle if the inclusion map $Hol \mathcal{R}\hookrightarrow Sec \mathcal{R}$ induces a epimorphism at $0$-th homotopy group $\pi_0$. $\mathcal{R}$ satisfies parametric $h$-principle if $\pi_k(Sec \mathcal{R}, Hol \mathcal{R})=0$ for all $k\geq 0$.\\ 
  
Let $(V,d\alpha_V)$ and $(W,d\alpha_W)$ be two exact symplectic manifolds with $dim(V)\leq dim(W)$. $\mathcal{R}_{isosymp}$ be the relation of isosymplectic immersions from $V$ to $W$. Let $A\subset V$ be a polyhedron of positive codimension.

\begin{theorem}(\cite{Mishachev})
\label{isosymp}
All forms of local $h$-principle holds for the inclusion \[Hol_{OpA}\mathcal{R}_{isosymp}\to Sec_{OpA}\mathcal{R}_{isosymp}\] on $OpA$.
\end{theorem}
 
 \begin{proof}
 Direct consequence of 16.4.2 of \cite{Mishachev}.
 \end{proof}

\section{Donaldson's Construction with a Preassigned Singularity}
\label{DC}
In this section we follow the construction in \cite{Donaldson} with some changes in order to construct a sequence of sections $s_k:W^{2n}\to E^k$ which will have a non-degenerate singularity on $L\subset W$, where $L,W$ are as in \ref{Main} and $E$ is a hermitian line bundle on $W$ and $E^k=E^{\otimes k}$.\\

Now we shall recall some parts of the local theory as in \cite{Donaldson}.\\

A complex structure on a real vector space $V$ is a decomposition of $V^*\otimes \mathbb{C}$ as a complex conjugate subspaces, the complex linear and anti-linear functionals. If a reference complex structure is fixed then $V$ becomes $\mathbb{C}^n$ and in this case the conjugate subspaces turns out to be just one forms of type $(1,0)$ and $(0,1)$, i.e, \[Hom_{\mathbb{R}}(\mathbb{C}^n,\mathbb{C})=\Lambda^{1,0}\oplus \Lambda^{0,1}\]  If $J$ is another complex structure then $\Lambda^{1,0}_J$ is the graph of a complex linear map \[\mu:\Lambda^{1,0}\to \Lambda^{0,1}\] and in this case $\Lambda^{0,1}_J$ is the graph of $\bar{\mu}:\Lambda^{0,1}\to \Lambda^{1,0}$. The condition that $\Lambda^{1,0}_J \pitchfork \Lambda^{0,1}_J$ is that $(1-\mu\bar{\mu})$ is invertible. For an almost-complex structure $J$ the corresponding $\mu$ will be a bundle map i.e, it depends on the base point smoothly and we shall denote it by $\mu_z$ where $z$ is the base point. In this case \[\bar{\partial}_J(f)=\bar{\partial}f-\mu (\partial f),\] where $\partial,\ \bar{\partial}$ are the ordinary operators defined by the standard complex structure on $\mathbb{C}^n$. \\

For $\rho <1$ define $\delta_{\rho}:\mathbb{C}^n\to \mathbb{C}^n$ by $\delta_{\rho}(z)=\rho z$. Consider the almost-complex structure $\tilde{J}=\delta_{\rho}^*J$ defined on $\rho^{-1}\Omega$ where $\Omega$ is the open neighborhood of zero in $\mathbb{C}^n$ where $J$ was defined. The $\tilde{\mu}$ corresponding to $\tilde{J}$ is given by \[\tilde{\mu}_z=\mu_{\rho z}\] In this case we get the following estimate \[|\tilde{\mu}_z|\leq C\rho|z|,\ |\triangledown \tilde{\mu}|\leq C\rho,\] where $C$ is a constant related to the Nijenhius tensor and its derivative. Throughout the paper we use the convention that $C$ represents a positive constant which changes from line to line.\\

Now consider the K\"{a}hler form \[\omega_0=\frac{i}{2}\Sigma_{j=1}^{n}dz_{j}d\bar{z}_{j}\] on $\mathbb{C}^n$. Set $A=\frac{1}{4}(\Sigma_{j=1}^{n}z_{j}d\bar{z}_{j}-\bar{z}_{j}dz_{j})$ then $\omega_0=idA$. \\

If $\xi \to \mathbb{C}^n$ be a line bundle with connection having curvature $\omega_0$ there is a trivialization of $\xi$ with respect to which the connection matrix turns out to be $A$. $A$ defines a coupled $\bar{\partial}$-operator $\bar{\partial}_A$, \[\bar{\partial}_A(f)=\bar{\partial}f+A^{0,1}f,\] where $A^{0,1}$ is the $(0,1)$ component of $A$. Similarly we define $\partial_A$. Now observe that 
\begin{enumerate}
\item $\bar{\partial}_A(e^{-(Re\Sigma z_j^2)/4}+i)=\frac{e^{-(Re\Sigma z_j^2)/4}}{4}\Sigma_j (z_j-\bar{z}_j)d\bar{z}_j+\frac{i}{4}\Sigma_jz_jd\bar{z}_j$\\
\item $\partial_A(e^{-(Re\Sigma z_j^2)/4}+i)=-\frac{e^{-(Re\Sigma z_j^2)/4}}{4}\Sigma_j (z_j+\bar{z}_j)dz_j-\frac{i}{4}\Sigma_j\bar{z}_jdz_j$\\
\end{enumerate}

We replace $\omega_0$ by $k\omega_0$, for a positive integer $k$. The replacement of $\xi$ by $\xi^k$ is same as the dilation with scalar factor $k^{-1/2}$ on $\mathbb{C}^n$.\\

Now let $(W,\omega)$ be a \textit{Liouville} manifold with a compatible complex structure $J$ and a complex line bundle $E\to W$ with a $U(1)$ connection having curvature $-i\omega$. Let $g$ be the Riemannian metric determined by $J$ and $\omega$. Then $g_k=kg$ is the one determined by $J$ and $k\omega$.\\

We consider the cubical coordinate chart which is longer in the $y$-coordinate directions, i.e, for $R>>r$ set \[I^{2n}_{r,R}=\{(x,y)\in \mathbb{R}^n\times \mathbb{R}^n:|x_j|<r,|y_j|<R\}\] In fact we take $R=k^{-1/2}$ and $r=k^{-1}$. Consider the Darboux chart $\chi_{p,k}:I^{2n}_{k^{-1}+\epsilon_k,k^{-1/2}+\epsilon_k}\to W$ around $p\in W$, where $\epsilon'_k=\epsilon k^{-1}\ and\ \epsilon_k=\epsilon k^{-1/2}$. We may further assume that all derivatives of $\chi_p$ are bounded and derivative of $\chi_p$ at zero is complex linear with respect to $J$ on $T_pW$. Therefore $\chi^*J$ represents a structure determined by a bundle map $\mu$ over $I^{2n}_{k^{-1}+\epsilon'_k,k^{-1/2}+\epsilon'_k}$ as explained above. The derivatives of $\mu$ satisfies bounds independent of $p\in W$. So now given $k$ we consider the new chart\[\tilde{\chi}_{p,k}=\chi_{p,k}\circ \delta_{k^{-1/2}}:k^{1/2}I^{2n}_{k^{-1}+\epsilon'_k,k^{-1/2}+\epsilon'_k}=I^{2n}_{k^{-1/2}+\epsilon_k,1+\epsilon_k}\to W\] Let $\tilde{\mu}$ be the bundle map representing the almost-complex structure on this new chart. So $\tilde{\mu}$ satisfies the following bounds
\[
\begin{array}{lcl}
|\tilde{\mu}_z|&\leq & Ck^{-1/2}|z|,\\
|\triangledown \tilde{\mu}_z|&\leq & Ck^{-1/2}.\\
\end{array} 
\]
$C$ does not depend on $p$. Observe that $\tilde{\chi}^*(-ik\omega)=-i\omega_0$ where $-ik\omega$ is the curvature on $E^k$. So $\tilde{\chi}$ admits a connection preserving bundle map which we shall also denote by $\tilde{\chi}$. So $\tilde{\chi}:\xi\to E^k$ is the bundle map. So one can think of $\sigma=e^{-(Re\Sigma z_j^2)/4}+i$ as local sections of $E^k$ around $p\in W$.\\

Let $\bar{\partial}_{A,\tilde{J}}$ denote the $\bar{\partial}$-operator defined by $\tilde{J}$ on $k^{1/2}I^{2n}_{k^{-1/2}+\epsilon_k,1+\epsilon_k}$ with connection matrix $A$. Then \[\bar{\partial}_{A,\tilde{J}}(f)=(\bar{\partial}f+A^{0,1}f)+\tilde{\mu}(\partial f+A^{1,0}f).\] If we replace $f$ by $e^{-(Re\Sigma z_j^2)/4}+i$ we get the following estimate \[|\bar{\partial}_{A,\tilde{J}}(e^{-(Re\Sigma z_j^2)/4}+i)|\leq \frac{|z|}{4}(2e^{-(Re\Sigma z_j^2)/4}+1)+Ck^{-1/2}|z|^2(2e^{-(Re\Sigma z_j^2)/4}+1)\] Furthermore we want the estimate on the derivative. So observe \[\triangledown=\bar{\partial}_{A,\tilde{J}}+\partial_{A,\tilde{J}}=\bar{\partial}_A+\partial_A+\tilde{\mu}\partial_A+\bar{\tilde{\mu}}\bar{\partial}_A\] The only term in $\bar{\partial}_{A,\tilde{J}}$ for which $\tilde{\mu}$ is not a factor is \[\bar{\partial}_A(e^{-(Re\Sigma z^2_j)/4}+i)=\frac{e^{-(Re\Sigma z^2_j)/4}}{4}\Sigma_j(z_j-\bar{z}_j)d\bar{z}_j+\frac{i}{4}\Sigma_jz_jd\bar{z}_j\] So let us compute $(\bar{\partial}_A+\partial_A)(\bar{\partial}_A(e^{-(Re\Sigma z^2_j)/4}+i))$. So \\
\[
\begin{array}{rcl}
(\bar{\partial}_A+\partial_A)(\bar{\partial}_Af) &=& (\bar{\partial}_A+\partial_A)(\bar{\partial}f+A^{0,1}f)\\
 &=& \bar{\partial}_A(\partial f+A^{0,1}f)+\partial_A(\partial f+A^{0,1}f)\\
  &=& (\bar{\partial}^2f+\bar{\partial}A^{0,1}f+A^{0,1}\bar{\partial}f+A^{0,1}A^{0,1}f)\\
  & & +(\partial\bar{\partial}f+\partial A^{0,1}f+A^{1,0}\bar{\partial}f+A^{1,0}A^{0,1}f)
\end{array}
\]
The first term \\
\[(1)\ \bar{\partial}^2f=0\]
The second term \\
\[
\begin{array}{rcl}
(2)\ \bar{\partial}A^{0,1}f &=& \bar{\partial}(\frac{1}{4}\Sigma_j(e^{-(Re\Sigma z^2_j)/4}+i)z_jd\bar{z}_j)\\
\Rightarrow (1/2)(\frac{\partial}{\partial_{x_l}}+i\frac{\partial}{\partial_{y_l}})A^{0,1}f &=& \frac{1}{8}[e^{-(Re\Sigma z^2_j)/4}+\Sigma_jz_j\frac{e^{-(Re\Sigma z^2_j)/4}}{4}(-2x_l)\\
 & & +\frac{i}{4}+i[ie^{-(Re\Sigma z^2_j)/4}+\Sigma_jz_j\frac{e^{-(Re\Sigma z^2_j)/4}}{4}(2y_l)+(i^2/4)]]d\bar{z}_l\wedge d\bar{z}_j\\
 \Rightarrow \bar{\partial}A^{0,1}f &=& - \frac{e^{-(Re\Sigma z^2_j)/4}}{4^2}\Sigma_{j,l}\bar{z}_l z_j d\bar{z}_l \wedge d\bar{z}_j
\end{array}
\]

The third term 

\[
\begin{array}{rcl}
(3)A^{0,1}\bar{\partial}f &=& A^{0,1}[-\Sigma_j z_j \frac{e^{-(Re\Sigma z^2_j)/4}}{4}d\bar{z}_j]\\
 &=& -\frac{e^{-(Re\Sigma z^2_j)/4}}{4^2}\Sigma_{j,l} z_l\bar{z}_j d\bar{z}_l\wedge d\bar{z}_j
\end{array}
\]

The forth term

\[
\begin{array}{rcl}
(4)A^{0,1}A^{0,1}f &=& \frac{e^{-(Re\Sigma z^2_j)/4}}{4^2}\Sigma_{j,l}z_lz_jd\bar{z}_l\wedge d\bar{z}_j\\
 & & +\frac{i}{4^2}\Sigma_{j,l}z_l z_j d\bar{z}_l\wedge d\bar{z}_j
\end{array}
\]

The fifth term

\[
\begin{array}{rcl}
(5)\partial \bar{\partial}f &=& \partial [-\Sigma_jz_j\frac{e^{-(Re\Sigma z^2_j)/4}}{4}d\bar{z}_j]\\
\frac{1}{2}(\frac{\partial}{\partial_{x_l}}-i\frac{\partial}{\partial_{y_l}})(\bar{\partial}f) &=& (1/2)[-\frac{e^{-(Re\Sigma z^2_j)/4}}{4}-\Sigma_jz_j\frac{e^{-(Re\Sigma z^2_j)/4}}{4^2}(-2x_l)\\
 & & +i[i\frac{e^{-(Re\Sigma z^2_j)/4}}{4}+\Sigma_jz_j \frac{e^{-(Re\Sigma z^2_j)/4}}{4^2}(2y_l)]]dz_l\wedge d\bar{z}_j\\
 \partial \bar{\partial}f &=& [-\frac{e^{-(Re\Sigma z^2_j)/4}}{4}+\frac{e^{-(Re\Sigma z^2_j)/4}}{4^2}\Sigma_{j,l}z_lz_j]dz_l\wedge d\bar{z}_j 
\end{array}
\]

The Sixth term

\[
\begin{array}{rcl}
(6)\partial A^{0,1}f &=& \partial[(1/4)\Sigma_jz_j(e^{-(Re\Sigma z^2_j)/4}+i)d\bar{z}_j]\\
(1/2)(\frac{\partial}{\partial_{x_l}}-i\frac{\partial}{\partial_{y_l}})A^{0,1}f &=& (1/8) [(e^{-(Re\Sigma z^2_j)/4}+i)+\Sigma_jz_j\frac{e^{-(Re\Sigma z^2_j)/4}}{4}(-2x_l)\\
 & & -i[i(e^{-(Re\Sigma z^2_j)/4}+i)+\Sigma_jz_j\frac{e^{-(Re\Sigma z^2_j)/4}}{4}(2y_l)]]dz_l\wedge d\bar{z}_j\\
 \partial A^{0,1}f &=& [(1/4)(e^{-(Re\Sigma z^2_j)/4}+i)-\frac{e^{-(Re\Sigma z^2_j)/4}}{4^2}\Sigma_{j,l}z_jz_l]dz_l\wedge d\bar{z}_j
\end{array}
\]

The seventh term

\[
\begin{array}{rcl}
(7)A^{1,0}\bar{\partial}f &=& A^{1,0}[-\Sigma_jz_j\frac{e^{-(Re\Sigma z^2_j)/4}}{4}d\bar{z}_j]\\
 &=& \Sigma_{j,l}\bar{z}_lz_j \frac{e^{-(Re\Sigma z^2_j)/4}}{4^2}dz_l\wedge d\bar{z}_j 
\end{array}
\]

The eighth term 

\[
\begin{array}{rcl}
(8)A^{1,0}A^{0,1}f &=& -(1/4^2)\Sigma_{j,l}\bar{z}_l z_j f dz_l\wedge d\bar{z}_j\\
 &=& -\frac{e^{-(Re\Sigma z^2_j)/4}}{4^2}\Sigma_{j,l} \bar{z}_lz_j dz_l\wedge d\bar{z}_j-(i/4^2)\Sigma_{j,l} \bar{z}_lz_jdz_l\wedge d\bar{z}_j
\end{array}
\]

So we get 

\[
\begin{array}{rcl}
(\partial_A+\bar{\partial}_A)(\bar{\partial}_Af) &=&  - \frac{e^{-(Re\Sigma z^2_j)/4}}{4^2}\Sigma_{j,l}\bar{z}_l z_j d\bar{z}_l \wedge d\bar{z}_j\\
 & &  -\frac{e^{-(Re\Sigma z^2_j)/4}}{4^2}\Sigma_{j,l} z_l\bar{z}_j d\bar{z}_l\wedge d\bar{z}_j\\
 & & +\frac{e^{-(Re\Sigma z^2_j)/4}}{4^2}\Sigma_{j,l}z_lz_jd\bar{z}_l\wedge d\bar{z}_j\\
 & & +\frac{i}{4^2}\Sigma_{j,l}z_l z_j d\bar{z}_l\wedge d\bar{z}_j\\
 & & +[-\frac{e^{-(Re\Sigma z^2_j)/4}}{4}+\frac{e^{-(Re\Sigma z^2_j)/4}}{4^2}\Sigma_{j,l}z_lz_j]dz_l\wedge d\bar{z}_j \\
 & & +[(1/4)(e^{-(Re\Sigma z^2_j)/4}+i)-\frac{e^{-(Re\Sigma z^2_j)/4}}{4^2}\Sigma_{j,l}z_jz_l]dz_l\wedge d\bar{z}_j\\
 & & +\Sigma_{j,l}\bar{z}_lz_j \frac{e^{-(Re\Sigma z^2_j)/4}}{4^2}dz_l\wedge d\bar{z}_j \\
 & & -\frac{e^{-(Re\Sigma z^2_j)/4}}{4^2}\Sigma_{j,l} \bar{z}_lz_j dz_l\wedge d\bar{z}_j-(i/4^2)\Sigma_{j,l} \bar{z}_lz_jdz_l\wedge d\bar{z}_j\\
 &=& \frac{e^{-(Re\Sigma z^2_j)/4}}{4^2}[\Sigma_{j,l}(z_lz_j-z_l\bar{z}_j-\bar{z}_lz_j)]d\bar{z}_l\wedge d\bar{z}_j\\
 & & +\frac{i}{4^2}[\Sigma_{j,l}z_lz_jd\bar{z}_l\wedge d\bar{z}_j-\Sigma_{j,l}\bar{z}_lz_j dz_l\wedge d\bar{z}_j]\\
 & & +\frac{i}{4} dz_l\wedge d\bar{z}_j
\end{array}
\]

Using $|z_j|\leq |z|$ we observe 

\[\frac{e^{-(Re\Sigma z^2_j)/4}}{4^2}[\Sigma_{j,l}(z_lz_j-z_l\bar{z}_j-\bar{z}_lz_j)]d\bar{z}_l\wedge d\bar{z}_j\leq \frac{e^{-(Re\Sigma z^2_j)/4}}{4} \frac{3}{4}{n\choose 2}|z|^2\]

\[\frac{i}{4^2}[\Sigma_{j,l}z_lz_jd\bar{z}_l\wedge d\bar{z}_j-\Sigma_{j,l}\bar{z}_lz_j dz_l\wedge d\bar{z}_j]\leq \frac{1}{8}{n\choose 2}|z|^2\]

\[\frac{i}{4} dz_l\wedge d\bar{z}_j\leq \frac{2}{8}{n\choose 2}\]

So the estimate we need is (using $|\tilde{\mu}_z|\leq Ck^{-1/2}|z|$)
\[
\begin{array}{rcl}
|\triangledown\bar{\partial}_{A,\tilde{J}}(e^{-(Re\Sigma z_j^2)/4}+i)| &\leq& \frac{e^{-(Re\Sigma z^2_j)/4}}{4} \frac{3}{4}{n\choose 2}|z|^2+\frac{1}{8}[{n\choose 2}|z|^2+2{n\choose 2}]\\
 & & +Ck^{-1/2}[e^{-(Re\Sigma z^2_j)/4}P^3(|z|)+Q^3(|z|)]\\
\end{array}
\]
where $P^3,Q^3$ are cubic polynomial with constant term zero.\\

We now define the cut-off function $\beta_k:\mathbb{C}^n\to \mathbb{R}$ as a $C^{\infty}$-function with $\beta_k=s_k$ on $I^{2n}_{k^{-1/2}-\epsilon_k,1-\epsilon_k}$ and $\beta_k=0$ outside $I^{2n}_{k^{-1/2}+\epsilon_k,1+\epsilon_k}$, where $s_k$ is a positive real sequence so small such that derivative $\triangledown \beta_k$ satisfies $|\triangledown \beta_k|\leq Ck^{-1}|z|^2$ and hence $|\triangledown \triangledown \beta_k|\leq Ck^{-1}|z|$. So we have \[|\bar{\partial}_{A,\tilde{J}}(\beta_k f)|\leq |\beta_k| |\bar{\partial}_{A,\tilde{J}}f|+|\triangledown\beta_k||f|\] Similarly \[|\triangledown(\bar{\partial}_{A,\tilde{J}}\beta_k f)|\leq |\triangledown \triangledown \beta_k||f|+|\triangledown \beta_k||\triangledown_Af|+|\beta_k||\triangledown(\bar{\partial}_{A,\tilde{J}}f)|+ |\triangledown \beta_k||\bar{\partial}_{A,\tilde{J}}f|\] So the cut-off function improves the estimates as follows.

\[
\begin{array}{rcl}
|\bar{\partial}_{A,\tilde{J}}\beta_k(e^{-(Re\Sigma z_j^2)/4}+i)|&\leq& |\beta_k|[\frac{|z|}{4}(2e^{-(Re\Sigma z_j^2)/4}+1)+Ck^{-1/2}|z|^2(2e^{-(Re\Sigma z_j^2)/4}+1)]\\
 & & + |\triangledown \beta_k|(e^{-(Re\Sigma z_j^2)/2}+1)^{1/2}\\
  &\leq & \frac{1}{2} s_k|z|e^{-(Re\Sigma z_j^2)/4}+|\beta_k|\frac{|z|}{4}+Ck^{-1/2}s_k|z|^2 e^{-(Re\Sigma z_j^2)/4}\\
 & &+Ck^{-1/2}|\beta_k||z|^2 + Ck^{-1}|z|^2e^{-(Re\Sigma z_j^2)/2}+|\triangledown \beta_k|
\end{array}
\]

\[
\begin{array}{rcl}
|\triangledown \bar{\partial}_{A,\tilde{J}}\beta_k(e^{-(Re\Sigma z_j^2)/4}+i)|&\leq& |\triangledown \triangledown \beta_k| (e^{-(Re\Sigma z_j^2)/2}+1)^{1/2}\\
 & &+|\triangledown \beta_k|\frac{|z|}{4}(2e^{-(Re\Sigma z_j^2)/4}+1)\\
 & &+|\beta_k| [\frac{e^{-(Re\Sigma z^2_j)/4}}{4} \frac{3}{4}{n\choose 2}|z|^2+\frac{1}{8}[{n\choose 2}|z|^2+2{n\choose 2}]]\\
  & &+|\beta_k|[Ck^{-1/2}[e^{-(Re\Sigma z^2_j)/4}P^3(|z|)+Q^3(|z|)]]\\
 & &+ Ck^{-1}|z|^2|\bar{\partial}_{A,\tilde{J}}f|\\
 &\leq & |\triangledown \triangledown \beta_k| e^{-(Re\Sigma z_j^2)/2}+|\triangledown \triangledown \beta_k|\\
 & &+|\triangledown \beta_k|\frac{|z|}{4}e^{-(Re\Sigma z_j^2)/4}+|\triangledown \beta_k|\frac{|z|}{4}\\
 & &+s_k[\frac{e^{-(Re\Sigma z^2_j)/4}}{4} \frac{3}{4}{n\choose 2}|z|^2+\frac{1}{8}[{n\choose 2}|z|^2+2{n\choose 2}]]\\
  & &+s_kCk^{-1/2}e^{-(Re\Sigma z^2_j)/4}P^3(|z|)+Ck^{-1/2}|\beta_k|Q^3(|z|)\\
 & &+ Ck^{-1}|z|^2|\bar{\partial}_{A,\tilde{J}}f|
\end{array}
\]

We define the smooth section $\sigma_p$ of the line bundle $E^k$ over $W$ by \[\sigma_p=(\beta_k\sigma)\circ \tilde{\chi}^{-1}\] around $p\in \tilde{\chi}(I^{2n}_{k^{-1/2}+\epsilon_k,1+\epsilon_k})$ which we extend by $0$ on all of $W$.\\

Now fix a finite cover, independent of $k$, of $W$ by charts $\phi_s:O_s\to W$, $s=1,...,S$, where $O_s\subset \mathbb{C}^n$ is bounded with \[\frac{1}{2}|x-y|\leq d(\phi(x),\phi(y))\leq 2|x-y|\] We may choose nested sets $O''_s\subset O'_s\subset O_s$, so that $W$ is covered by $\phi_s(O''_s)$. Let $\Lambda$ be the lattice $\alpha_k\mathbb{Z}^n+i\mathbb{Z}^n$ in $\mathbb{C}^n$, where $\alpha_k$ is chosen in such a way that the translations of the center of $I^{2n}_{k^{-1/2}+\epsilon_k,1+\epsilon_k}$ to the lattice points of $\Lambda$ cover $\mathbb{C}^n$. Let $\Lambda_s$ be $\phi_s(\Lambda\cap O'_s)$. Define the set $\cup_1^M\{p_i\}=\cup_1^S\Lambda_s$. So the images under $\phi_s$ of the translations of $I^{2n}_{k^{-1/2}+\epsilon_k,1+\epsilon_k}$ to the lattice points of $\Lambda_s$ gives neighborhoods of $p_i$'s which cover $W$. Moreover we arrange that if a neighborhood of $p_i$ as mentioned above intersects the regular Lagrangian $L\subset W$ then $p_i\in L$. Later we shall translate the origins of these coordinate charts except one (the first one) such that there is only one singularity on $L$.\\

Now consider the terms of the right hand sides of the estimates for \[|\bar{\partial}_{A,\tilde{J}}\beta_k(e^{-(Re\Sigma z_j^2)/4}+i)|\ and\ |\triangledown \bar{\partial}_{A,\tilde{J}}\beta_k(e^{-(Re\Sigma z_j^2)/4}+i)|\] with the understanding that outside of the neighborhoods of $p_i$'s they vanish as $\beta_k$ vanishes outside these neighborhoods. We also note that we can replace $s_k$ by $k^{-1/2}s_k$ for large $k$.\\

 Our final section $s$ will be \[s=s_{w}=\Sigma_1^Mw_i\sigma_i,\ \sigma_i=\sigma_{p_i},\ |w_i|\leq 1,\ w=(w_1,...,w_M)\] Obviously it depends on $k$. Later we shall make some specific choice for $w_i$'s. So we have the following

\begin{theorem}
\label{sec-est}
For any choice of co-efficient vector $w$ with $|w_i|\leq 1$ and a very small choice of $s_k$ for large $k$, the section $s=s_w$ satisfies 
\[
\begin{array}{rcl}
|s| &\leq & C\\
|\bar{\partial}_E s|&\leq & C(k^{-1/2}+k^{-1})\\
|\triangledown_W\bar{\partial}_E s| &\leq & C(k^{-1/2}+k^{-1})\\
\end{array}
\]
\end{theorem}

Now we shall recall a result from \cite{Donaldson}.

\begin{theorem}(\cite{Donaldson})
\label{Don-symp}
If $a',a'':\mathbb{C}^n\to \mathbb{C}$ are respectively complex linear and anti-linear maps and if $|a''|<|a'|$, then the subspace $ker(a'+a'')\subset \mathbb{C}^n$ is symplectic.
\end{theorem}

In view of \ref{sec-est} and \ref{Don-symp} if we prove the following result then we shall prove that the zero locus of $s$ is a symplectic submanifold as in \cite{Donaldson}.

\begin{theorem}
\label{fin-est}
There is an $\epsilon>0$ such that for all large $k$ we can choose $w$ with $|w_i|\leq 1$, so that $s=s_w$ satisfies the transversality condition \[|\partial s|>\epsilon\] on the zero locus of $s$ i.e $\mathcal{Z}(s)$.
\end{theorem}

\begin{remark}
We shall further need that $w_i=c+ic,\ or\ -c+ic$ for $c>0$ and shall also see that such a choice is possible in this context.
\end{remark}

Now we shall prove \ref{fin-est} which similar to \cite{Donaldson} (not identical though).

\begin{theorem}(\cite{Donaldson})
\label{index}
Given any $D>0$ there is a number $N=N(D)$, independent of $k$, such that for any large $k$ we can choose a collection of centers $p_i$ satisfying the conditions of the proof of \ref{sec-est} in addition with the property that there is a partition of the set $\Lambda=\{1,...,M\}$ into $N$ disjoint subsets $\Lambda=\Lambda_1\cup ...\cup\Lambda_N$ such that for each $\alpha$ \[d(p_i,p_j)\geq D\ if\ p_i,p_j\in \Lambda_{\alpha}\]
\end{theorem}

So now we assume that $D$ is fixed and $p_i$'s accordingly as well as the partition $\{\Lambda _{\alpha}\}$. The choice of $D$ will be made at the end of this section. We think of this partition $\{\Lambda_{\alpha}\}$ as the colouring of the cubes $I^{2n}_{k^{-1/2}+\epsilon_k,1+\epsilon_k}$.\\

We shall adjust now the co-efficients $w_i$'s belonging to the same $\Lambda_{\alpha}$ starting with $\Lambda_1$ with the co-efficient vector $w_0=(w_1,0,...,0),\ w_1=c+ic,\ c>0$ where $p_1\in L$. Here it is different from \cite{Donaldson}. In the next stage we shall change the co-efficients $w_i$'s belonging to $\Lambda_2$ leaving the others unchanged and proceed this way. Thus the total number $N$ of stages in this construction depends on $D$ but not on $k$.\\

Now we can focus on the adjustment criterion for $w_i$'s. The goal is to achieve the controlled transversality over the cubes belonging to $\Lambda_{\alpha}$ at the $\alpha$-th stage. We set \[W_{\alpha}=\cup_{i\in \Lambda_{\beta},\beta \leq \alpha} I^{2n}_i\] where $I^{2n}_i$ is the $i$-th cube  $I^{2n}_{k^{-1/2}+\epsilon_k,1+\epsilon_k}$. So $\Phi(empty)=W_0\subset ...\subset W_N=W$. We want that for suitable $\epsilon$, the section $s_{\alpha}$ satisfies $|\partial s_{\alpha}|>\epsilon$ on $\mathcal{Z}(s_{\alpha})\cap W_{\alpha}$. So the criterion can be summarized as 
\begin{enumerate}
\item The change in the co-efficient belonging to $\Lambda_{\alpha}$ must achieve controlled transversality over the cubes belonging to $\Lambda_{\alpha}$.\\
\item The change in the co-efficients belonging to $\Lambda_{\alpha}$ must not destroy the controlled transversality that has already been achieved over the cubes belonging to $\Lambda_{\beta}$ for $\beta < \alpha$.\\ 
\end{enumerate}

So the precise form of transversality we need is the following 
\begin{definition}(\cite{Donaldson})
\label{transversality}
Let $f:U\to \mathbb{C}$ be a smooth map on an open set $U\subset \mathbb{C}$ and let $w\in \mathbb{C}$. For $\eta >0$ we say that $f$ is $\eta$-transverse to $w$ over $U$ if for any $z\in U$ such that $|f(z)-w|\leq \eta$ the derivative satisfies $|(\partial f)_z|\geq \eta$.
\end{definition}

\ref{transversality} is stable under $C^1$-perturbations, moreover if $f$ is $\eta$-transverse to $w$ and $g:U \to \mathbb{C}$ is such that $|f-g|_{C^1}\leq \delta$ then $g$ is $\eta-\delta$-transverse to $w$.\\

  Now we consider any section $s$ of $E^k\to W$. On a fixed chart around say $p_i$, $E^k\to W$ has a standard trivialization given by $\sigma_i$. We write $s=f_i\sigma_i$, for a function $f_i$ defined on the cube $I^{2n}_{k^{-1/2}+\epsilon_k,1+\epsilon_k}$. We say that the section $s$ is $\eta$-transverse on the chart if $f_i$ is $\eta$-transverse to $0$ on the cube $I^{2n}_{k^{-1/2},1}$.
  
  \begin{lemma}
  \label{f-est-1}
  If $s=s_{w}$ is the section of $E^k$ as described in \ref{sec-est} and $f_i$'s the corresponding function then 
  \begin{enumerate}
  \item $|f_i|_{C^1(I^{2n}_{k^{-1/2},1})}\leq C$\\
  \item $|\bar{\partial}f_i|_{C^1(I^{2n}_{k^{-1/2},1})}\leq C(k^{-1/2}+k^{-1})$\\
  \item If $|\partial f_i|>\epsilon$ on $f_i^{-1}(0)\cap I^{2n}_{k^{-1/2},1}$, then for large $k$, $|\partial_E s|>C^{-1}\epsilon$ on the intersection of $\mathcal{Z}(s)$ and the chart.
  \end{enumerate}
  \end{lemma}
  In view of \ref{sec-est}, the proofs of \ref{f-est-1} is same as the proofs of Lemma-18 from \cite{Donaldson}.
  
  \begin{lemma}
  \label{f-est-2}
  If $s=s_w$ be as in \ref{sec-est}, for any $\alpha$, let $w'$ be another vector which agrees with $w$ except for the co-efficients belonging to $\Lambda_{\alpha}$ and suppose that $|w'_j-w_j|\leq \delta$ for $j\in \Lambda_{\alpha}$. Set $s'=s_{w'}$ and let $f_i,f'_i$ be the functions representing these sections, then for large $D$
  \begin{enumerate}
  \item For any $i$, $|f'_i-f_i|_{C^1(I^{2n}_{k^{-1/2},1})}\leq \delta$\\
  \item If $i\in \Lambda_{\alpha}$ and $w'_i=w_i+\theta_i$ then $|f'_i-f_i-\theta_i|_{C^1(I^{2n}_{k^{-1/2},1})}=0$\\
\end{enumerate}   
  \end{lemma} 
  
  \begin{proof}
  We write $s=\Sigma_jw_j\sigma_j$. So on $i$-th smaller cube  $I^{2n}_{k^{-1/2},1}$ our equation $s=f_i\sigma_i$ gives us $\Sigma_{j\neq i}w_j\sigma_j+w_i\sigma_i=f_i\sigma_i$ and similarly $\Sigma_{j\neq i}w_j'\sigma_j+w_i'\sigma_i=f'_i\sigma_i$. So we get 
\[  
  \begin{array}{rcl}
 f'_i-f_i-\theta_i &=& w'_i-w_i-\theta_i+ \Sigma_{j\neq i}\frac{\sigma_j}{\sigma_i}(w'_j-w_j)\\
\Rightarrow |f'_i-f_i-\theta_i| &\leq & \Sigma_{j\neq i}|w'_j-w_j||\frac{\sigma_j}{\sigma_i}|,\ As\ w_i'-w_i-\theta_i=0\\  
  \end{array}
  \]
  Similarly $|f'_i-f_i| \leq \delta+ \Sigma_{j\neq i}|w'_j-w_j||\frac{\sigma_j}{\sigma_i}|$. Now considering the derivative we get 
\[
\begin{array}{rcl}
\triangledown (f'_i-f_i-\theta_i)&=&\Sigma_{j\neq i}(w'_j-w_j)\triangledown (\frac{\sigma_j}{\sigma_i})\\
&=& \Sigma_{j\neq i}(w'_j-w_j)[\frac{1}{\sigma_i}\triangledown \sigma_j -\frac{\sigma_j}{\sigma_i^2}\triangledown \sigma_i]\\
\Rightarrow |\triangledown (f'_i-f_i-\theta_i)| &\leq &\Sigma_{j\neq i} |w_j'-w_j|[\frac{|\triangledown \sigma_j|}{|\sigma_i|}+\frac{|\sigma_j|}{|\sigma_i^2|}|\triangledown \sigma_i|]\\
\end{array}
\]  
  Now as $w'_j=w_j$ for $j \notin \Lambda_{\alpha}$, the summations $\Sigma_{j\neq i}$ in the above estimates can be replaced by $\Sigma_{j\in \Lambda_{\alpha}}$. So when $D$ is large $|\sigma_j|=0=|\triangledown \sigma_j|$ on the chart $\tilde{\chi}_i(I^{2n}_{k^{-1/2},1})$.  
  \end{proof}
  
Now define $Q_p(\delta)=\log(\delta^{-1})^{-p},\ \delta>0$. The next result (similar to Theorem-20 of \cite{Donaldson}) gives us the suitable $w$ in order to get the transversality condition. The result is also known as Quantitative Transversality theorem. We shall however show that we can make specific choice for $w_j$'s in our context which will be needed in the proof of \ref{Main}. 

\begin{theorem}(\cite{Donaldson})
\label{QTT}
For $\sigma >0$, let $\mathcal{H}_{\sigma}$ denote the set of functions $f$ on $I^{2n}_{k^{-1/2},1}$ such that 
\begin{enumerate}
\item $|f|_{C^0(I^{2n}_{k^{-1/2},1})}\leq 1$\\
\item $|\bar{\partial} f|_{C^0(I^{2n}_{k^{-1/2},1})}\leq \sigma$\\
\end{enumerate}
Then there is an integer $p$, depending only on the dimension $n$, such that for any $\delta$ with $0<\delta<1/2$, if $\sigma\leq Q_p(\delta)\delta$, then for any $f\in \mathcal{H}_{\sigma}$ there is a $w\in \mathbb{C}$ with $|w|\leq \delta$ such that $f$ is $Q(\delta)\delta$-transverse to $w$ over $I^{2n}_{k^{-1/2},1}$. Moreover $w$ can be taken to be $-(a+ia),\ a>0$ and $-(a+ib),\ b=-a>0$. The choice between these two options will be clarified later.
\end{theorem}

We shall prove \ref{QTT} at the end of this section. For now let us continue with our construction.\\

Let us assume that we have reached the $\alpha-1$-th stage of our construction and we want to reach the $\alpha$-th stage. So we have chosen the $w^{\alpha-1}$ so that $s_{\alpha-1}$ is $\eta_{\alpha}$-transverse over $W_{\alpha-1}\subset W$ for some positive $\eta_{\alpha-1}$ with $0<\eta_{\alpha-1}<\rho$. We will now choose $w_i^{\alpha},\ i\in \Lambda_{\alpha}$. We need $|w_i^{\alpha}-w_i^{\alpha-1}|\leq \delta_{\alpha}$. We shall set $\delta_{\alpha}>0$ shortly. With this choice of $w_i^{\alpha}$, \ref{f-est-2} gives us that $s^{\alpha}$ is $\eta_{\alpha-1}-\delta_{\alpha}$-transverse over $W_{\alpha-1}$. So if we set $\delta_{\alpha}=\frac{1}{2}\eta_{\alpha-1}$ we can conclude that $s^{\alpha}$ is still $\frac{1}{2}\eta_{\alpha-1}$-transverse on $W_{\alpha-1}$. Now we shall consider the cubes $I^{2n}_i,\ i\in \Lambda_{\alpha}$. Here the section $s^{\alpha-1}$ is represented by the function $f_i=f_i^{\alpha-1}=\frac{s^{\alpha-1}}{\sigma_i}$. By \ref{f-est-1} $f_i$ is bounded by a fixed constant $C$ over $I^{2n}_{k^{-1/2},1}$ and $C^{-1}f_i\in \mathcal{H}_{\sigma}$ for $\sigma=C^{-1}(k^{-1/2}+k^{-1})$. Now we apply \ref{QTT} $C^{-1}f_i$ for suitably small $\rho$ and $\delta=C^{-1}\delta_{\alpha}$ as long as \[k^{-1/2}+k^{-1}\leq C\delta_{\alpha}Q_p(\delta_{\alpha})\] And thus we get $v_i,\ i\in \Lambda_{\alpha}$ with $|v_i|\leq \delta_{\alpha}$ such that $f_i$ is $Q_p(\delta_{\alpha})\delta_{\alpha}$-transverse to $v_i$ and equivalently $f_i-v_i$ is $Q_p(\delta_{\alpha})\delta_{\alpha}$-transverse to $0$. Observe that $f_i-v_i$ represents $s_{w'}$ where $w'_j=w_j^{\alpha-1},\ if\ j\neq i$ and $w'_i=w_i^{\alpha-1}-v_i$. In view of \ref{f-est-2} for large $D$ we can make all the changes to the $w_i^{\alpha-1},\ i\in \Lambda_{\alpha}$ simultaneously as explained above. This concludes the proof of \ref{fin-est}.\\

Now we shall see the proof of \ref{QTT} which is same as the proof of Theorem-20 of \cite{Donaldson}. Although we shall outline the entire proof for completeness but all we need to show that we can make the choice for $w_i$'s in our context. The choice is different from \cite{Donaldson}.\\

Let $P:\mathbb{R}^n\to \mathbb{R}$ be a polynomial function of degree $d$. Let $S\subset \mathbb{R}^n$ be the subset \[S=\{x\in \mathbb{R}^n:|x|\leq 1,\ P(x)\leq 1\}\] Similarly set $S(\theta)=\{x:|x|\leq 1,\ P(x)\leq 1+\theta\}$. 

\begin{proposition}(\cite{Donaldson})
\label{decomposition}
There are constants $C,\ \nu$, depending only on the dimension, such that for any polynomial $P$ there are arbitrarily small positive $\theta$ so that $S$ can be decomposed into pieces \[S=S_1\cup...\cup S_A\] where $A\leq Cd^{\nu}$, in such a way that any pair of points in the same piece $S_r$ can be joined by a path in $S(\theta)$ of length less than $Cd^{\nu}$.
\end{proposition}
 The proof decomposes into two parts. First we prove the \ref{QTT} for holomorphic $f$ and then use some estimates to prove the result for approximately holomorphic $f$. 

  \begin{lemma}(\cite{Donaldson})
  \label{poly-Approx}
  Let $f:I^{2n}_{k^{-1/2},1}\to \mathbb{C}$ be a holomorphic map with $|f(z)|\leq 1$. Then for any $\epsilon$ satisfying $0<\epsilon<1/2$ there is a complex polynomial function $g$ on $\mathbb{C}^n$ of degree less than $C\log \epsilon^{-1}$ such that $|f(z)-g(z)|,\ |\partial f-\partial g|\leq \epsilon$ in the interior region. 
  \end{lemma} 
  Now we shall prove \ref{QTT} for holomorphic $f$. So let us assume that $f$ is as in \ref{poly-Approx} and we use \ref{poly-Approx} to approximate $f$ by a polynomial $g$ of degree $d\leq C\log \epsilon^{-1}$. Define \[S^f=\{z\in I^{2n}_{k^{-1/2},1} \cup \mathbb{C}^n:|\partial f|\leq \epsilon\}\] \[S^g=\{z\in I^{2n}_{k^{-1/2},1} \cup \mathbb{C}^n:|\partial g|\leq 2\epsilon\}\] Obviously $S^f \subset S^g$ and therefore $f(S^f)\subset f(S^g)\subset N_{\epsilon}(g(S^g))$ where $N_{\epsilon}(g(S^g))$ is the $\epsilon$-neighborhood of $g(S^{g})$. Observe that $S^g$ is a semi-algebraic set of the kind in \ref{decomposition}, where we have taken $P=(2\epsilon)^{-2}|\partial g|^2$ whose degree is $2(d-1)$. So \ref{decomposition} gives a decomposition of $S^g$ into $A$ pieces. Let $z_1,z_2$ are in the same piece of $S^g$ then integrating the derivative of $g$ over a path of length less than $Cd^{\nu}$ in $P(\theta)$, for suitable $\theta$, joining $z_1$ to $z_2$, we have $|g(z_1)-g(z_2)|\leq 4\epsilon Cd^{\nu}$. Therefore $g(S^g)$ is contained in the union of $A$ discs in $\mathbb{C}$ each of radius $4\epsilon Cd^{\nu}$. So $f(S^g)$ and hence $f(S^f)$ is contained in the union of slightly larger discs of radius $\epsilon(2Cd^{\nu}+1)$. \\
  
  The $\epsilon$-transversality between $f$ and $w$ is that $w$ lies outside the $\epsilon$-neighborhood of $f(S^f)$ and the neighborhood is contained in the union of $A$ discs of radius $\epsilon(4Cd^{\nu}+2)$. The total area in $\mathbb{C}$ of these discs is at most $A\pi (4Cd^{\nu}+2)^2\epsilon^2$. If we choose $\rho$ such that the area of the half disc \[\Omega=\{w\in \mathbb{C}: |w|\leq \rho,\ Im(w)\leq 0\}\] is biggere than the total area covered by these discs , there is a $w\in \Omega$ not contained in the $\epsilon$-neighborhood of $f(S^f)$. The condition on $\rho$ is \[\frac{1}{2}\pi \rho^2>A\pi(4Cd^{\nu}+2)^2\epsilon^2\] which is $\rho>(2A)^{1/2}(4Cd^{\nu}+2)\epsilon$. Observe that $A$ are bounded by powers of degree of $P$, hence of the degree $d$ of $g$ which is bounded by a power of $\log(\epsilon)^{-1}$. \\
  
  So over all, for arbitrarily small $\epsilon$ there exists a $w\in \mathbb{C}$ such that $f$ is $\epsilon$-transverse to $w$ and $|w|\leq C\epsilon \log(\epsilon^{-1})^p$ for some $p$. Now one needs to re-arrange the parameters. For fixed $C,p$, the function $h$ given by $h(\epsilon)=C\epsilon\log(\epsilon^{-1})^p$ is increasing for for small $\epsilon$ and tends to $0$ as $\epsilon$ approaches $0$. If $\eta=h(\epsilon)$ then \[\eta \log(\eta^{-1})^{-p}=C\epsilon (\frac{\log\epsilon^{-1}}{\log \epsilon^{-1}-p\log \log \epsilon^{-1}-\log C})^p <2C\epsilon,\ (if\ \epsilon\ small)\] By inverting $h$, we conclude that for small $\eta$ there is a $w$ with $|w|\leq \eta$ such that $f$ is $\frac{1}{2C}\eta (\log \eta^{-1})^p$-transverse to $w$, and by increasing $p$ and assuming $\eta$ to be small, we can replace the factor $\frac{1}{2C}$ by $1$.\\
  
  In the above we considered holomorphic $f$. Now we shall consider approximately holomorphic $f$. Although we have not shown that we can make the specific choice for $w$ in our context, but we shall show this at the end of the proof. \\

Using H\"{o}rmander's weighted $L^2$ space method \cite{Hormander} one gets the following

\begin{lemma}(\cite{Donaldson}) 
\label{Hor}
For each $r<1$ there is a constant $K=K(r)$ such that if $f$ is any smooth complex valued function on $I^{2n}_{k^{-1/2},1}$, then there is a holomorphic function $\tilde{f}$ on $rI^{2n}_{k^{-1/2},1}$ such that \[|f-\tilde{f}|_{C^1(rI^{2n}_{k^{-1/2},1})}\leq K(|\bar{\partial}f|_{C^0(I^{2n}_{k^{-1/2},1})}+|\triangledown \bar{\partial}f|_{C^0(I^{2n}_{k^{-1/2},1})})\]
\end{lemma}

Now if $f\in \mathcal{H}_{\sigma}$, we approximate $f$ by a holomorphic function $\tilde{f}$ on a slightly smaller region with $|f-\tilde{f}|_{C^1}\leq C\sigma$. Then use the above arguments for $\tilde{f}$.\\

Now we show how we can make the specific choice for $w$ in the half disc \[\{w \in \mathbb{C}:|w|\leq \rho,\ Im(w)\leq 0\}\] We must show that $f-v_i$ is $\tilde{\epsilon}$-transverse to $0$ for some small $\tilde{\epsilon}>0$. Any coordinate chart intersect only finitely many adjacent charts and the number depends on the dimension $n$. We only consider one such intersection. Let $\phi=(\phi_1,...,\phi_n)$ be the transition function. So we take \[f=\frac{\beta'[(ce^{-(Re\Sigma \phi_j^2)/4}-d)+i(de^{-(Re\Sigma \phi_j^2)/4}+c)]}{\beta(e^{-(Re\Sigma z_j^2)/4}+i)}\] where either $c=d>0$ or $d=-c>0$. Here $\beta$ corresponds to the chart under consideration and $\beta'$ corresponds to the adjacent chart. The sign of $c,d$ changes as we are adding $-v_i\sigma_i$ because $w_i'=w_i^{\alpha}-v_i$ and at the beginning we set $w_0=(w_1,0,..,0)$ where $w_1=c+ic,\ c>0$. Simplifying we get 
\[
\begin{array}{rcl}
f &=&\frac{\beta'}{\beta(e^{-(Re\Sigma z_j^2)/2}+1)} [\{c(1+e^{-(Re\Sigma z_j^2)/4}e^{-(Re\Sigma \phi_j^2)/4})+d(e^{-(Re\Sigma \phi_j^2)/4}-e^{-(Re\Sigma z_j^2)/4})\}\\
 & & +i\{c(e^{-(Re\Sigma z_j^2)/4}-e^{-(Re\Sigma \phi_j^2)/4})+d(1+e^{-(Re\Sigma z_j^2)/4}e^{-(Re\Sigma \phi_j^2)/4})\}]\\ 
\end{array}
\]

Now observe that for large $k$, i.e, for range of $|x_j|$ being very small $(Assuming\ z_j=x_j+iy_j)$, the imaginary part of $f$ for both $c=d>0$ and $d=-c>0$ is non-negative. To see this consider the numerator of the imaginary part for $c=d>0$ 
\[
c[(e^{-(Re\Sigma z_j^2)/4}-e^{-(Re\Sigma \phi_j^2)/4})+(1+e^{-(Re\Sigma z_j^2)/4}e^{-(Re\Sigma \phi_j^2)/4})]\]
\[
=c[1+e^{-(Re\Sigma z_j^2)/4}(1+e^{-(Re\Sigma \phi_j^2)/4})-e^{-(Re\Sigma \phi^2)/4}]\]
First observe that if $e^{-(Re\Sigma z_j^2)/4}>1$ then the imaginary part is non-negative.\\

If $e^{-(Re\Sigma z_j^2)/4}<1$ we observe that as $k$ becomes large $e^{-(Re\Sigma z_j^2)/4}$ becomes very close to $1$. So for large $k$ the imaginary part is non-negative.\\

We get the same result for the case $d=-c>0$ by making $k$ large enough.\\

 So the imaginary part of $f-v_i$ is bigger than $a$ if we take $v_i=-(a+ia),\ a>0$ and is bigger than $b$ if we take $v_i=-(a+ib),\ b=-a>0$ which concludes the argument. For more intersection we may need to take larger $k$.\\
 
 Now we make an observation which will be used in the next section. 
 
 \begin{remark}
 \label{translation}
 In the construction above we can make translation say $y_j+constant$ in the coordinate charts (which are cubes). To see this notice that this will only need a readjustment in the constant $C$ in \ref{sec-est} which were adjusted multiple times during the construction as mentioned above.
 \end{remark}

\section{Main Theorem} In this section we prove \ref{Main}. As in the previous section we consider $s$ for large values of $k$. So $\mathcal{Z}(s)$ is a symplectic submanifold and hence $s^{-1}(z)$ for $z$ close to $0$ is also symplectic. This forms a tubular neighborhood of $\mathcal{Z}(s)$. \\

Now we consider the $p_1\in L$ the center of the first chart in the construction of $s$ as in the previous section. Obviously it is a singularity for $s$ and now we'll figure out the stable manifold corresponding to the vector field $X_{Im \log s}$, the Hamiltonian vector field corresponding to $Im \log s=Arg(s)$ (similar idea has been used in \cite{Pardon} as we have mentioned in the introduction). But before we do so let us show that it is a hyperbolic zero for $X_{Im \log s}$.

\begin{proposition}
\label{Hyperbolic}
$p_1\in L\subset W$ is a hyperbolic zero for $X_{Im \log s}=X_{\theta},\ where\ \theta=Arg(s)$ and the tangent space to the stable manifold is given by $\{x_j=y_j:j=1,...,n\}$ at $p_1$. Moreover on the set $\beta_1=constant\neq 0$ the stable manifold is given by $\{x_j=y_j:j=1,...,n\}$. 
\end{proposition}

\begin{proof}
On the set $\beta_1=constant\neq 0$, $s$ is given by $s=\beta_1 (a+ib)(e^{-(Re\Sigma z_j^2)/4}+i)$. So we get 
\[
\begin{array}{rcl}
\tan \theta &=& \frac{be^{-(Re\Sigma z_j^2)/4}+a}{ae^{-(Re\Sigma z_j^2)/4}-b}=\frac{A}{B}\\
\Rightarrow \sec^2 \theta d\theta &=& (1/B)[(1/2)be^{-(Re\Sigma z_j^2)/4}\Sigma_j(y_jdy_j-x_jdx_j)]\\
 & & -(A/B^2)[(1/2)ae^{-(Re\Sigma z_j^2)/4}\Sigma_j(y_jdy_j-x_jdx_j)]\\
 \Rightarrow d\theta &=& \frac{B}{2(A^2+B^2)}[be^{-(Re\Sigma z_j^2)/4}\Sigma_j(y_jdy_j-x_jdx_j)]\\
 & & -\frac{A}{2(A^2+B^2)}[ae^{-(Re\Sigma z_j^2)/4}\Sigma_j(y_jdy_j-x_jdx_j)]
\end{array}
\]
The second line is achieved from the first by differentiating and the third line is achieved by multiplying by $\cos^2 \theta$. Note that because we are differentiating both sides of an equation we get $d_A=d$.\\

Now let $X_{\theta}=\Sigma_j(a_j\partial_{x_j}+b_j \partial_{y_j})$ and observe that the symplectic structure on the tubular neighborhood of $L\subset W$ is represented by $\omega_{st}=\Sigma_j dx_j\wedge dy_j$ on coordinate charts. Hence Hamiltonian equation $d\theta=\iota_{X_{\theta}}\omega_{st}$ gives us $d\theta=\Sigma_j (a_jdy_j-b_jdx_j)$. Solving we get \[a_j=\frac{y_je^{-(Re\Sigma z_j^2)/4}}{2(A^2+B^2)}[Bb-Aa]\] Similarly \[b_j=\frac{x_je^{-(Re\Sigma z_j^2)/4}}{2(A^2+B^2)}[Bb-Aa]\] Now we compute the partial derivatives of $a_j\ and\ b_j$ at zero. Observe that zero in this coordinates corresponds to $p_1$. So we get \[(\partial_{x_j}a_j)_{\mid 0}=0,\ (\partial_{y_j}a_j)_{\mid 0}=\frac{Bb-Aa}{2(A^2+B^2)}_{\mid 0}=-\frac{1}{4}\] Similarly \[(\partial_{y_j}b_j)_{\mid 0}=0,\ (\partial_{x_j}b_j)_{\mid 0}=\frac{Bb-Aa}{2(A^2+B^2)}_{\mid 0}=-\frac{1}{4}\] Define $F(z)=(a_1,b_1,...,a_n,b_n)$, then $DF_{\mid 0}$ has all eigenvalues equal to $\pm \frac{1}{4}$. The tangent of the stable manifold at zero corresponds to $-\frac{1}{4}$ which is given by $\{x_j=y_j:j=1,...,n\}$.\\

 Now let $(x(t),y(t))=(x_1(t),y_1(t),...,x_n(t),y_n(t))$ be the flow of $X_{\theta}$. So \[\partial_t x_j(t)=a_j(x(t),y(t))\ and\ \partial_t y_j(t)=b_j(x(t),y(t))\] Multiplying the first equation $\partial_t x_j(t)=a_j(x(t),y(t))$ by $x_j(t)$ and the second equation $\partial_t y_j(t)=b_j(x(t),y(t))$ by $y_j(t)$ and integrating we get (suppressing the $t$ all along below for convenient notations) \[\frac{1}{2}(x_j^2-y_j^2)=0\] which proves the remaining statement. Here we are discarding the $x_j=-y_j$ possibility by the description of the tangent space at zero.
\end{proof}

Now we shall figure out the stable manifold globally. But for this we need to figure out the flow of the vector field $X_{\theta}$. Let us consider $s$ on a chart adjacent to the one containing $p_1$ and intersecting $L$. Recall that on these charts $L$ is given by $\{y_j=0,\ j=1,...,n\}$. Now on this chart $\tan \theta$ is given by \[\tan \theta=\frac{\beta_1be^{-(Re\Sigma z_j^2)/4}+\beta_1a+\beta_2de^{-(Re\Sigma \phi_j^2)/4}+\beta_2c}{\beta_1ae^{-(Re\Sigma z_j^2)/4}-\beta_1b+\beta_2ce^{-(Re\Sigma \phi_j^2)/4}-\beta_2d}\] where $\beta_2$ corresponds to the chart containing $p_1$ and $\beta_1$ corresponds to the chart adjacent to it and $\phi=(\phi_1,...,\phi_n)$ is the transition function. Obviously $c=d>0$ but the choice of $a,b$ will be made clear shortly between the two options given by \ref{QTT}. Now we proceed as in the proof of \ref{Hyperbolic}.

\[
\begin{array}{rcl}
\tan \theta &=&\frac{\beta_1be^{-(Re\Sigma z_j^2)/4}+\beta_1a+\beta_2de^{-(Re\Sigma \phi_j^2)/4}+\beta_2c}{\beta_1ae^{-(Re\Sigma z_j^2)/4}-\beta_1b+\beta_2ce^{-(Re\Sigma \phi_j^2)/4}-\beta_2d}=\frac{A}{B}\\
\Rightarrow \sec^2 \theta d\theta &=& (1/B)[(1/2)b\beta_1e^{-(Re\Sigma z_j^2)/4}\Sigma_j(y_jdy_j-x_jdx_j)\\
 & & +be^{-(Re\Sigma z_j^2)/4}\Sigma_j(\partial_{x_j}\beta_1 dx_j+\partial_{y_j}\beta_1 dy_j)\\
 & & +\Sigma_j\partial_{x_j}(\beta_2 d e^{-(Re\Sigma \phi_j^2)/4})dx_j+\Sigma_j\partial_{y_j}(\beta_2 d e^{-(Re\Sigma \phi_j^2)/4})dy_j\\
  & & +\Sigma_j \partial_{x_j}(\beta_2 c+\beta_1 a)dx_j+\Sigma_j \partial_{y_j}(\beta_2 c+\beta_1 a)dy_j]\\
  & & -(A/B^2)[(1/2)a\beta_1e^{-(Re\Sigma z_j^2)/4}\Sigma_j(y_jdy_j-x_jdx_j)\\
 & & +ae^{-(Re\Sigma z_j^2)/4}\Sigma_j(\partial_{x_j}\beta_1 dx_j+\partial_{y_j}\beta_1 dy_j)\\
 & & +\Sigma_j \partial_{x_j}(\beta_2 ce^{-(Re\Sigma \phi_j^2)/4})dx_j+\Sigma_j \partial_{y_j}(\beta_2 c e^{-(Re\Sigma \phi_j^2)/4})dy_j\\
& & -\Sigma_j \partial_{x_j}(\beta_1 b+\beta_2 d)dx_j-\Sigma_j \partial_{y_j}(\beta_1 b+\beta_2 d)dy_j]\\
\Rightarrow d\theta &=& \frac{B}{A^2+B^2}[(1/2)b\beta_1e^{-(Re\Sigma z_j^2)/4}\Sigma_j(y_jdy_j-x_jdx_j)\\
 & & +be^{-(Re\Sigma z_j^2)/4}\Sigma_j(\partial_{x_j}\beta_1 dx_j+\partial_{y_j}\beta_1 dy_j)\\
 & & +\Sigma_j\partial_{x_j}(\beta_2 d e^{-(Re\Sigma \phi_j^2)/4})dx_j+\Sigma_j\partial_{y_j}(\beta_2 d e^{-(Re\Sigma \phi_j^2)/4})dy_j\\
  & & +\Sigma_j \partial_{x_j}(\beta_2 c+\beta_1 a)dx_j+\Sigma_j \partial_{y_j}(\beta_2 c+\beta_1 a)dy_j]\\
  & & -\frac{A}{A^2+B^2}[(1/2)a\beta_1e^{-(Re\Sigma z_j^2)/4}\Sigma_j(y_jdy_j-x_jdx_j)\\
 & & +ae^{-(Re\Sigma z_j^2)/4}\Sigma_j(\partial_{x_j}\beta_1 dx_j+\partial_{y_j}\beta_1 dy_j)\\
 & & +\Sigma_j \partial_{x_j}(\beta_2 ce^{-(Re\Sigma \phi_j^2)/4})dx_j+\Sigma_j \partial_{y_j}(\beta_2 c e^{-(Re\Sigma \phi_j^2)/4})dy_j\\
& & -\Sigma_j \partial_{x_j}(\beta_1 b+\beta_2 d)dx_j-\Sigma_j \partial_{y_j}(\beta_1 b+\beta_2 d)dy_j]\\
\end{array}
\]
Again as in \ref{Hyperbolic} let $X_{\theta}=\Sigma_j(a_j\partial_{x_j}+b_j\partial_{y_j})$. Then solving the Hamiltonian equation we get 

\[
\begin{array}{rcl}
a_j &=&\frac{B}{A^2+B^2}[(1/2)by_j \beta_1e^{-(Re\Sigma z_j^2)/4}+be^{-(Re\Sigma z_j^2)/4}\partial_{y_j}\beta_1\\
& & + d\partial_{y_j}(\beta_2e^{-(Re\Sigma \phi_j^2)/4})+\partial_{y_j}(\beta_2 c+\beta_1 a)]\\
 & & -\frac{A}{A^2+B^2}[(1/2)ay_j\beta_1 e^{-(Re\Sigma z_j^2)/4} + ae^{-(Re\Sigma z_j^2)/4}\partial_{y_j}\beta_1\\
 & & +c\partial_{y_j}(\beta_2 e^{-(Re\Sigma \phi_j^2)/4})-\partial_{y_j}(\beta_2 d+\beta_1 b)]\\
 &=& \frac{y_j\beta_1e^{-(Re\Sigma z_j^2)/4}}{2(A^2+B^2)}[Bb-Aa]+\frac{e^{-(Re\Sigma z_j^2)/4}}{A^2+B^2}[Bb-Aa]\partial_{y_j}\beta_1\\
 & & +\frac{Bd-Ac}{A^2+B^2}\partial_{y_j}(\beta_2e^{-(Re\Sigma \phi_j^2)/4})+ \frac{Bc+Ad}{A^2+B^2}\partial_{y_j}\beta_2\\
 & & +\frac{aB+bA}{A^2+B^2}\partial_{y_j}\beta_1\\
\end{array}
\]
Similarly we get $b_j$. Now let $(x(t),y(t))=(x_1(t),y_1(t),...,x_n(t),y_n(t))$ be the flow of $X_{\theta}$. So \[\partial_t x_j(t)=a_j(x(t),y(t))\ and\ \partial_t y_j(t)=b_j(x(t),y(t))\] Multiplying the first equation $\partial_t x_j(t)=a_j(x(t),y(t))$ by $x_j(t)$ and the second equation $\partial_t y_j(t)=b_j(x(t),y(t))$ by $y_j(t)$ and integrating we get (suppressing the $t$ all along below for convenient notations)

\[
\begin{array}{rcl}
\frac{1}{2}(x_j^2-y_j^2) &=& \int \frac{bB-aA}{A^2+B^2}e^{-(Re\Sigma z_j^2)/4}(x_j\partial_{y_j}-y_j\partial_{x_j})(\beta_1)dt \\
 & & +\int \frac{Bd-Ac}{A^2+B^2}(x_j\partial_{y_j}-y_j\partial_{x_j})(\beta_2e^{-(Re\Sigma \phi_j^2)/4})dt\\
 & & +\int \frac{Bc+Ad}{A^2+B^2}(x_j\partial_{y_j}-y_j\partial_{x_j})(\beta_2)dt\\
 & & +\int \frac{Ba+Ab}{A^2+B^2}(x_j\partial_{y_j}-y_j\partial_{x_j})(\beta_1)dt\\
\end{array}
\]

First we simplify the integrands. Consider the first term. So \[Bb-Aa=-\beta_1(a^2+b^2)+ \beta_2(bc-ad)e^{-(Re\Sigma \phi_j^2)/4}-\beta_2(bd+ac)\] Now consider the forth term. So \[aB+bA=\beta_1(a^2+b^2)e^{-(Re\Sigma z_j^2)/4}+\beta_2(ac+bd)e^{-(Re\Sigma \phi_j^2)/4}+\beta_2(bc-ad)\] Hence 
\[
\begin{array}{rcl}
(Bb-Aa)e^{-(Re\Sigma z_j^2)/4}+(aB+bA)&=& \beta_2(bc-ad)(e^{-(Re\Sigma \phi_j^2)/4}e^{-(Re\Sigma z_j^2)/4}+1)\\
& & +\beta_2(ac+bd)(e^{-(Re\Sigma \phi_j^2)/4}-e^{-(Re\Sigma z_j^2)/4})\\
&=&2bc \beta_2 (e^{-(Re\Sigma \phi_j^2)/4}e^{-(Re\Sigma z_j^2)/4}+1)
\end{array}
\]

The last line achieved by putting $c=d>0$ and $b=-a>0$. Similarly considering the second term we get 

\[
\begin{array}{rcl}
Bd-Ac &=& \beta_1(ad-bc)e^{-(Re\Sigma z_j^2)/4}-\beta_1(bd+ac)-\beta_2(c^2+d^2)\\
&=& -2bc\beta_1e^{-(Re\Sigma z_j^2)/4}-2\beta_2c^2
\end{array}
\]
Again the last line is achieved by putting $c=d>0$ and $b=-a>0$. Now consider the third term. 

\[
\begin{array}{rcl}
Bc+Ad &=& \beta_1(ac+bd)e^{-(Re\Sigma z_j^2)/4}-\beta_1(bc-ad)+\beta_2(c^2+d^2)e^{-(Re\Sigma \phi_j^2)/4}\\
&=& -2bc\beta_1+2\beta_2c^2 e^{-(Re\Sigma \phi_j^2)/4},\ (putting\ c=d>0,\ b=-a>0)\\
\end{array}
\]

Observe that inside the chart, i.e, when $\beta_2=0$ and $\beta_1=constant\neq 0$, all the integrands in the equation of the flow vanishes. So the equation becomes $x_j^2-y_j^2=\Xi=constant$. We shall now make the constant on the right side in the resulting equation of the flow negative for large values of $k$ and for suitable translation of the origin $0$ in this chart.\\

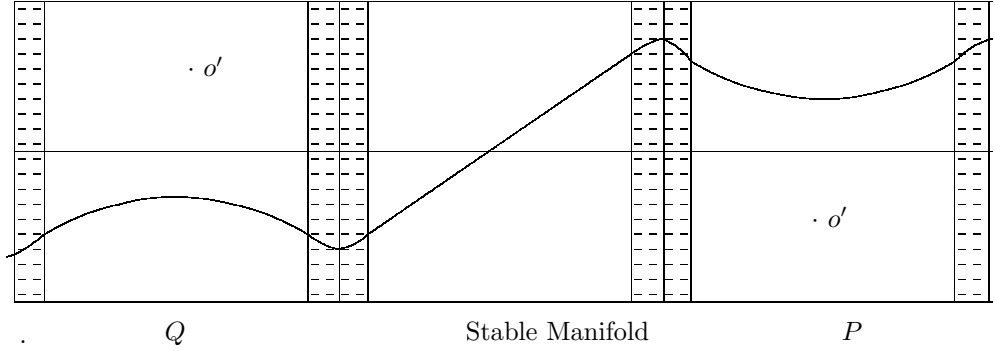
 \begin{center}
\begin{picture}(300,150)(-100,5)\setlength{\unitlength}{1cm}
\linethickness{.075mm}

\multiput(-5,0)(4.32,0){4}
{\line(0,1){4}}

\multiput(-5,0)(3.9,0){2}
{\line(0,1){4}}

\multiput(-5,0)(4.7,0){2}
{\line(0,1){4}}

\multiput(-5,0)(9,0){2}
{\line(0,1){4}}

\multiput(-5,0)(8.2,0){2}
{\line(0,1){4}}

\multiput(-5,0)(12.5,0){2}
{\line(0,1){4}}

\multiput(-5,0)(.4,0){2}
{\line(0,1){4}}

\multiput(-5,.001)(0,2){3}
{\line(1,0){13}}


\qbezier(1.3,2)(3.2,3.3)(3.2,3.3)
\qbezier(1.3,2)(-.3,.9)(-.3,.9)

\qbezier(3.2,3.3)(3.4,3.45)(3.6,3.5)
\qbezier(3.6,3.5)(3.8,3.45)(4,3.2)

\qbezier(4,3.2)(4.5,2.9)(5,2.8)
\qbezier(5,2.8)(5.8,2.6)(6.5,2.8)
\qbezier(6.5,2.8)(7,2.9)(7.5,3.2)

\qbezier(7.5,3.2)(7.75,3.45)(8,3.5)

\qbezier(-.3,.9)(-.45,.75)(-.7,.7)
\qbezier(-.7,.7)(-.9,.75)(-1.1,.9)

\qbezier(-1.1,.9)(-1.6,1.2)(-2.1,1.3)
\qbezier(-2.1,1.3)(-2.9,1.5)(-3.6,1.3)
\qbezier(-3.6,1.3)(-4.1,1.2)(-4.6,.9)

\qbezier(-4.6,.9)(-4.9,.65)(-5.1,.6)


\multiput(-1.05,3.9)(.2,0){4}{\line(1,0){.09}}
\multiput(-1.05,3.7)(.2,0){4}{\line(1,0){.09}}
\multiput(-1.05,3.5)(.2,0){4}{\line(1,0){.09}}
\multiput(-1.05,3.3)(.2,0){4}{\line(1,0){.09}}
\multiput(-1.05,3.1)(.2,0){4}{\line(1,0){.09}}
\multiput(-1.05,2.9)(.2,0){4}{\line(1,0){.09}}
\multiput(-1.05,2.7)(.2,0){4}{\line(1,0){.09}}
\multiput(-1.05,2.5)(.2,0){4}{\line(1,0){.09}}
\multiput(-1.05,2.3)(.2,0){4}{\line(1,0){.09}}
\multiput(-1.05,2.1)(.2,0){4}{\line(1,0){.09}}
\multiput(-1.05,1.9)(.2,0){4}{\line(1,0){.09}}
\multiput(-1.05,1.7)(.2,0){4}{\line(1,0){.09}}
\multiput(-1.05,1.5)(.2,0){4}{\line(1,0){.09}}
\multiput(-1.05,1.3)(.2,0){4}{\line(1,0){.09}}
\multiput(-1.05,1.1)(.2,0){4}{\line(1,0){.09}}
\multiput(-1.05,.9)(.2,0){4}{\line(1,0){.09}}
\multiput(-1.05,.7)(.2,0){4}{\line(1,0){.09}}
\multiput(-1.05,.5)(.2,0){4}{\line(1,0){.09}}
\multiput(-1.05,.3)(.2,0){4}{\line(1,0){.09}}
\multiput(-1.05,.1)(.2,0){4}{\line(1,0){.09}}

\multiput(3.25,3.9)(.2,0){4}{\line(1,0){.09}}
\multiput(3.25,3.7)(.2,0){4}{\line(1,0){.09}}
\multiput(3.25,3.5)(.2,0){4}{\line(1,0){.09}}
\multiput(3.25,3.3)(.2,0){4}{\line(1,0){.09}}
\multiput(3.25,3.1)(.2,0){4}{\line(1,0){.09}}
\multiput(3.25,2.9)(.2,0){4}{\line(1,0){.09}}
\multiput(3.25,2.7)(.2,0){4}{\line(1,0){.09}}
\multiput(3.25,2.5)(.2,0){4}{\line(1,0){.09}}
\multiput(3.25,2.3)(.2,0){4}{\line(1,0){.09}}
\multiput(3.25,2.1)(.2,0){4}{\line(1,0){.09}}
\multiput(3.25,1.9)(.2,0){4}{\line(1,0){.09}}
\multiput(3.25,1.7)(.2,0){4}{\line(1,0){.09}}
\multiput(3.25,1.5)(.2,0){4}{\line(1,0){.09}}
\multiput(3.25,1.3)(.2,0){4}{\line(1,0){.09}}
\multiput(3.25,1.1)(.2,0){4}{\line(1,0){.09}}
\multiput(3.25,.9)(.2,0){4}{\line(1,0){.09}}
\multiput(3.25,.7)(.2,0){4}{\line(1,0){.09}}
\multiput(3.25,.5)(.2,0){4}{\line(1,0){.09}}
\multiput(3.25,.3)(.2,0){4}{\line(1,0){.09}}
\multiput(3.25,.1)(.2,0){4}{\line(1,0){.09}}

\multiput(7.55,3.9)(.2,0){2}{\line(1,0){.09}}
\multiput(7.55,3.7)(.2,0){2}{\line(1,0){.09}}
\multiput(7.55,3.5)(.2,0){2}{\line(1,0){.09}}
\multiput(7.55,3.3)(.2,0){2}{\line(1,0){.09}}
\multiput(7.55,3.1)(.2,0){2}{\line(1,0){.09}}
\multiput(7.55,2.9)(.2,0){2}{\line(1,0){.09}}
\multiput(7.55,2.7)(.2,0){2}{\line(1,0){.09}}
\multiput(7.55,2.5)(.2,0){2}{\line(1,0){.09}}
\multiput(7.55,2.3)(.2,0){2}{\line(1,0){.09}}
\multiput(7.55,2.1)(.2,0){2}{\line(1,0){.09}}
\multiput(7.55,1.9)(.2,0){2}{\line(1,0){.09}}
\multiput(7.55,1.7)(.2,0){2}{\line(1,0){.09}}
\multiput(7.55,1.5)(.2,0){2}{\line(1,0){.09}}
\multiput(7.55,1.3)(.2,0){2}{\line(1,0){.09}}
\multiput(7.55,1.1)(.2,0){2}{\line(1,0){.09}}
\multiput(7.55,.9)(.2,0){2}{\line(1,0){.09}}
\multiput(7.55,.7)(.2,0){2}{\line(1,0){.09}}
\multiput(7.55,.5)(.2,0){2}{\line(1,0){.09}}
\multiput(7.55,.3)(.2,0){2}{\line(1,0){.09}}
\multiput(7.55,.1)(.2,0){2}{\line(1,0){.09}}

\multiput(-4.95,3.9)(.2,0){2}{\line(1,0){.09}}
\multiput(-4.95,3.7)(.2,0){2}{\line(1,0){.09}}
\multiput(-4.95,3.5)(.2,0){2}{\line(1,0){.09}}
\multiput(-4.95,3.3)(.2,0){2}{\line(1,0){.09}}
\multiput(-4.95,3.1)(.2,0){2}{\line(1,0){.09}}
\multiput(-4.95,2.9)(.2,0){2}{\line(1,0){.09}}
\multiput(-4.95,2.7)(.2,0){2}{\line(1,0){.09}}
\multiput(-4.95,2.5)(.2,0){2}{\line(1,0){.09}}
\multiput(-4.95,2.3)(.2,0){2}{\line(1,0){.09}}
\multiput(-4.95,2.1)(.2,0){2}{\line(1,0){.09}}
\multiput(-4.95,1.9)(.2,0){2}{\line(1,0){.09}}
\multiput(-4.95,1.7)(.2,0){2}{\line(1,0){.09}}
\multiput(-4.95,1.5)(.2,0){2}{\line(1,0){.09}}
\multiput(-4.95,1.3)(.2,0){2}{\line(1,0){.09}}
\multiput(-4.95,1.1)(.2,0){2}{\line(1,0){.09}}
\multiput(-4.95,.9)(.2,0){2}{\line(1,0){.09}}
\multiput(-4.95,.7)(.2,0){2}{\line(1,0){.09}}
\multiput(-4.95,.5)(.2,0){2}{\line(1,0){.09}}
\multiput(-4.95,.3)(.2,0){2}{\line(1,0){.09}}
\multiput(-4.95,.1)(.2,0){2}{\line(1,0){.09}}

\put(1,-.5){Stable Manifold}
\put(6,-.5){$P$}
\put(-3,-.5){$Q$}
\put(5.6,1){$\cdot\ o'$}
\put(-2.7,3){$\cdot\ o'$}

\end{picture}\end{center}

.\\
In the above picture first we consider the right side cube, i.e, $P$. The integrands are non-zero only on the shaded region in the picture. So we try to determine the signs of the integrands in this region.\\

Obviously $A^2+B^2\geq 0$. So consider the sum of first and forth term as above, i.e, the equation 
\[
\begin{array}{rcl}
(Bb-Aa)e^{-(Re\Sigma z_j^2)/4}+(aB+bA)&=& 2bc \beta_2 (e^{-(Re\Sigma \phi_j^2)/4}e^{-(Re\Sigma z_j^2)/4}+1)
\end{array}
\]
Obviously it is bigger than or equal to zero. So let us determine the sign of \[(x_j\partial_{y_j}-y_j\partial_{x_j})(\beta_1)\] Now observe that for large $k$ the range of $|x_j|$ is very small. So the sign of $-y_j\partial_{x_j}(\beta_1)$ determines the sign if we translate the origin as shown in the picture above which is possible by \ref{translation}. Observe that on the cube P (see the picture) the part of the shaded region of cube P touching the center cube, here  $\beta_1$ with respect to $x_j$ is increasing. Hence $\partial_{x_j}\beta_1\geq 0$. So $-y_j\partial_{x_j}(\beta_1)<0$ as on the intersection of the stable manifold and the shaded region, $y_j>0$. So this integrand is negative on the shaded region.\\

Now consider the second and the third integrands. The second integrand contains a factor \[(x_j\partial_{y_j}-y_j\partial_{x_j})(\beta_2 e^{-(Re\Sigma \phi_j^2)/4})=e^{-(Re\Sigma \phi_j^2)/4}(x_j\partial_{y_j}-y_j\partial_{x_j})(\beta_2)+\beta_2(x_j\partial_{y_j}-y_j\partial_{x_j})(e^{-(Re\Sigma \phi_j^2)/4})\] Observe that the third integrand contains a factor $(x_j\partial_{y_j}-y_j\partial_{x_j})(\beta_2)$. So we compute \[(Bd-Ac)(e^{-(Re\Sigma \phi_j^2)/4})+(Bc+Ad)=-2bc \beta_1 (e^{-(Re\Sigma \phi_j^2)/4}e^{-(Re\Sigma z_j^2)/4}+1)<0\] Now as above the sign is determined by $-y_j\partial_{x_j}\beta_2$. But this time  we have $\beta_2$ instead of $\beta_1$, which is decreasing on the shaded region. So $-y_j\partial_{x_j}\beta_2>0$. Hence the sign is negative. \\

But we are still left with one term, namely the one containing the factor \[\beta_2(x_j\partial_{y_j}-y_j\partial_{x_j})(e^{-(Re\Sigma \phi_j^2)/4})\] from the second integrand. Now we know that $Bd-Ac<0$. As $e^{-(Re\Sigma \phi_j^2)/4}$ is decreasing with respect to $x_j$, so if we take $k$ large and shift the origin in the $y_j$ component down enough we can make this integrand negative and hence we are done with the cube $P$. \\

So we have shown that the integrands are all negative in the shaded region of the cube $P$ which intersects the central cube in the picture. So we integrate from $t=0$ to $t=t_0$ where $t=0$ corresponds to $p_1$ the center of the central cube and $t=t_0$ corresponds to the point where the curve in the picture leaves the shaded region and enters the non-shaded region in the cube $P$. So the integral from $t=0$ to $t=t_0$ is negative. Observe that at $t=0$ i.e, $p_1$ in the central cube, all the integrands are zero.\\ 

We can do similar thing for the cube $Q$ and the global stable manifold is achieved by observing that in the above argument the role of $(a,b)$ and $(c,d)$ can be interchanged. \\

Let the stable manifold be denoted by $L'$. Now we shall show that the stable manifold $L'$ is diffeomorphic to the regular Lagrangian $L$, moreover $L$ and $L'$ are $C^0$-close. In the above picture the horizontal line represents $L$ and the curved line represents the stable manifold $L'$. First we need a lemma.

\begin{lemma}
\label{Control}
The constant $\Xi$ in the explanation of $L'$ above in the chart $P$, is negative and $|\Xi|$ can be made sufficiently small.
\end{lemma}

\begin{proof}
The fact that $\Xi$ is negative has been proved above. So we prove that $|\Xi|$ can be made sufficiently small. Observe that only the shaded region in the picture contribute to $\Xi$, as on the non-shaded region all the integrands are zero.\\

 We start with the first and forth term contributing in the constant $\Xi$ as above. It is \[2bc \beta_2 (e^{-(Re\Sigma \phi_j^2)/4}e^{-(Re\Sigma z_j^2)/4}+1)(x_j\partial_{y_j}-y_j\partial_{x_j})(\beta_1)\]
 
 $bc$ is coming from $w$ which is $|w|<\delta$ with $0<\delta<1/2$ (see \ref{QTT}) and $\beta_2\leq s_k$ ( here $\beta_2=(\beta_2)_k\leq s_k$, $k$ index of the sequence in section \ref{DC} and $s_k$ is not the sequence sections of the vector bundle) which becomes smaller and smaller as $k$ increases. Since we are translating the origin in the $y_j$ direction (\ref{translation}) and as we can make $|x_j|$ very small, we can do it in a way to make the modulus of the above term small.\\
 
 Same is applied to the term \[-2bc \beta_1 (e^{-(Re\Sigma \phi_j^2)/4}e^{-(Re\Sigma z_j^2)/4}+1)(x_j\partial_{y_j}-y_j\partial_{x_j})(\beta_2)\]
 
 So we consider the remaining term, namely \[\beta_2(x_j\partial_{y_j}-y_j\partial_{x_j})(e^{-(Re\Sigma \phi_j^2)/4})\] A suitable translation of the origin in the $y_j$ direction (\ref{translation}) along with the observation that $\beta_2$ is very small, will also make the modulus of this term small.

\end{proof}

\begin{proposition}
\label{Diff-C^0}
$L'$ is diffeomorphic to $L$ and is $C^0$-close to $L$.
\end{proposition} 

\begin{proof}
As explained above on the chart that contains $p_1$ (the center cube in the picture) the stable manifold is given by \[\{x_j=y_j: j=1...n\}\] So projection on the $x$-plane making $y_j$'s zero gives a diffeomorphism on this chart.\\

In the adjacent charts the stable manifold is given by \[\{x_j^2-y_j^2=\Xi=constant\}\] where the sign of the constant negative. Moreover $|\Xi|$ can be made small (see \ref{Control}). So the projection on $x$-plane defines a diffeomorphism on these charts. Notice that on the charts adjacent to these charts (i.e, charts like $P$ and $Q$ as in the picture) which intersects $L$ the description of the stable manifold $L'$ is given by \[\{x_j^2-y_j^2=\Xi=constant\}\] as well (See the explanation for $L'$ above and \ref{Control}). But this time the value of the constant $\Xi=\Xi_1$ (say) is determined by performing the integral from $t=t'_0$ to $t=t_1$ where $t=t_1$ corresponds to the similar point in this chart as $t=t_0$ in the cube $P$ and $t'_0$ corresponds to a point on the curve in the cube $P$ where all the integrands are zero, for example the point on the curve in the cube $P$ right over the center of the cube $P$. As in \ref{Control} we can make $|\Xi_1|$ very small. Continuing this way we can make all the $|\Xi_i|$ small.\\

 This means that in the picture $L'$ will continue in the adjacent charts of charts like $P$ and $Q$ as hyperbolic arcs as in $P$ and $Q$. So we can continue this way. Notice that $|\Xi|$ can be made small on these chart as in \ref{Control}, $L'$ never intersects the charts that does not intersect $L$, i.e, $L'$ never leaves the charts that form a tubular neighborhood of $L$, moreover for all charts $|x_j|$ are very small (See the section \ref{DC})and $|\Xi|$ small, in the description of $L'$ which shows that $L'$ is $C^0$-close to $L$.
\end{proof}

Now we construct the iso-symplectic immersion and show that it is possible to pull back the Weinstein structure by this iso-symplectic immersion.\\

\begin{proposition}
\label{iso-sym pull}
There exists an iso-symplectic immersion $F:W\to W$ which takes $L$ to $L'$, moreover it is possible to pull back the Weinstein structure by this iso-symplectic immersion with some modification near $\partial W$. In fact $F$ is homotopic through iso-symplectic immersions to identity and hence the pullback Weinstein structure (with some modification near $\partial W$) is homotopic through Weinstein structures (with some modification near $\partial W$) to the given Weinstein structure.
\end{proposition}

\begin{proof}
First let us construct the iso-symplectic immersion. We know there exists a diffeomorphism $\phi:L\to L'$. Since $L'$ is a Lagrangian submanifold of $W$, $\phi$ can be thought of as a Lagrangian immersion $\phi:L\to L'\subset W$. Moreover as $L$ is regular the underlying symplectic structure (of the Weinstein structure) of $W$ is given by that of $T^*L$ near $L$. So we can use \ref{WNT} to construct an iso-symplectic immersion $\phi':U\to W$ taking $L$ to $L'$, where $U$ is a tubular neighborhood of $L$ in $W$, i.e, to say of the zero section of $T^*L$. As $L'$ is $C^0$-close to $L$ we can assume that $L'\subset U$.\\

First extend $\phi'$ to all of $W$ as an iso-symplectic immersion. One way to do it is to consider $d\phi'=\Psi$ on $U$ and extend it to $W$ as a formal iso-symplectic monomorphism, i.e, $\Psi^*\omega=\omega$. Now extend $L$ (which is the given regular Lagrangian) into a polyhedron of $W$ of positive codimension (a core for $W$) say $A'$. Now we can use \ref{isosymp} (relative version) to show the existence of an iso-symplectic immersion $\psi$ on $Op(A')$ such that near $L$ it is equal to $\phi'$. As $W$ is a cobordism, hence open manifold by \ref{datta} there exists a homotopy of iso-symplectic immersions $g_t:W\to W$ such that $g_0=identity$ and $g_1(W)\subset Op(A')$. The required extension of $\phi'$ is $\psi\circ g_1$. This is the required $F$.\\

Since we are in an equi-dimensional  setup both the derivatives of identity and $F$ at some point $x\in W$ belongs to $Sp(2n,\mathbb{R})$ and since $Sp(2n,\mathbb{R})$ is path-connected we can join these derivatives by a path in $Sp(2n,\mathbb{R})$ but this path depends on $x\in W$. Thus we get a family of paths in $Sp(2n,\mathbb{R})$ joining the derivatives of identity and $dF$ parametrized by the points of $W$.\\

Let $\Phi_t$ be this path such that $\Phi_0=dF$ and $\Phi_1=d(id)$. So we use \ref{isosymp} on $Op(A')$  where $A'$ is as above. We get a path of iso-symplectic immersions say $\phi_t$ such that $\phi_0=\psi=F_{|Op(A')}$ and $\phi_1=identity$ on $Op(A')$. \\

Consider $\phi_t\circ g_1$, it is a path joining $F$ and $g_1$. If we concatenate this path with $g_{1-t}$, we get a path joining $F$ and $identity$ as $g_0=id$. Call this path $F_t:W\to W$, i.e, $F_t$ is a homotopy of iso-symplectic immersions ($F_t^*\omega=\omega$) such that $F_0=id_W$ and $F_1=F$.\\

Now we come to the second part of the statement. So we shall prove that we can pull back the Weinstein structure by $F$. Let $(\omega, X,\psi)$ be the given Weinstein structure. We shall define the pull back Weinstein structure under $F$.\\

As $F$ is iso-symplectic $F^*\omega=\omega$. Set $X'=(dF)^{-1}X$ and $\psi'=\psi\circ F$. First we notice that as $F$ is an immersion and we are in an equi-dimensional setup, $\psi'$ is a Morse function. To see this note that any immersion is a local diffeomorphism in an equi-dimensional setup by inverse function theorem.\\

 Secondly we shall see that $X'$ is a Liouville vectorfield. To see this let $\lambda$ be the Liouville form corresponding to $X$. Consider \[\omega_x((dF)^{-1}_{F(x)}(X_{F(x)}),Y_x)=F^*(\omega_{F(x)})((dF)^{-1}_{F(x)}(X_{F(x)}),Y_x)\] But $F^*(\omega_{F(x)})((dF)^{-1}_{F(x)}(X_{F(x)}),Y_x)$ is $\omega_{F(x)}(X_{F(x)},dF_x(Y_x))$ which is equal to \[\lambda_{F(x)}(dF_x(Y_x))=F^*(\lambda_x)(Y_x)\] Which shows that the Liouville form corresponding to $X'$ is $F^*\lambda$.\\

Notice that if $x'\neq x$ be such that $F(x')=F(x)$, then the above argument is applied to $x'$ as well. We are only using the fact that $dF_x:T_xW\to T_{F(x)}W$ is an isomorphism.\\

Now we show that $X'$ and $\psi'$ are gradient like. Since $X$ and $\psi$ are gradient like we have \[X.\psi \geq \delta (|X|^2+|d\psi|^2)\]

 So consider \[(X'.\psi')_x=(dF)^{-1}_{F(x)}(X_{F(x)}).(\psi\circ F)_x=d\psi_{F(x)}(dF_x (dF)^{-1}_{F(x)}(X_{F(x)}))=d\psi_{F(x)}(X_{F(x)})\] 
 
 But $d\psi_{F(x)}(X_{F(x)})\geq \delta_{F(x)}(|X_{F(x)}|^2+|d\psi_{F(x)}|^2)$. Choose a function $\delta'$ such that \[\delta_{F(x)}(|X_{F(x)}|^2+|d\psi_{F(x)}|^2)\geq \delta'_x(|(dF)^{-1}_{F(x)}(X_{F(x)})|^2+|d\psi_{F(x)}dF_x|^2)\] This choice of $\delta'$ makes $X'$ and $\psi'$ gradient like.\\
 
 Now we must modify the pull back Weinstein structure $\mathfrak{W}_1=(\omega,\psi',X')$ so that $X'\pitchfork \partial W$. So consider a tubular neighborhood $U_{\delta}$ of the boundary $\partial W$ in $W$ and set $\tilde{W}:=W-U_{\delta}$. Let $V_{\delta}$ be a tubular neighborhood of $\partial \tilde{W}$ in $\tilde{W}$.\\

  As above $F_t:W\to W$ is the homotopy of iso-symplectic immersions such that $F_0=id_W$ and $F_1=F$. So obviously $F_t^*\omega=\omega$ and as we have observed above $F_t^*\lambda$ is the Liouville form with respect to $\omega$ where $\lambda$ is the given Liouville form corresponding to $\omega$. This is because the exterior derivative and pull back commute so observe
 
  \[d(F_t^*\lambda)=F_t^*(d\lambda)=F_t^*\omega=\omega=d\lambda\]
 In the above we have used the fact that $\lambda$ is a Liouville form with respect to $\omega$ and $F_t:W\to W$ is iso-symplectic. Therefore $d(\lambda-F_t^*\lambda)=0$ 
and set $\lambda_t:= F_t^*\lambda$ so that

 \[\lambda_t:= F_t^*\lambda=\lambda+dh_t\] for some $h_t:W\to \mathbb{R}$ with $h_0=0$ using triviality of first cohomology of $W$.\\
 
Let $X$ be the Liouville vector field corresponding to $\lambda$. Then $X+X_{h_t}$ is the Liouville vector field corresponding to $\lambda_t$ where $X_{h_t}$ is the hamiltonian vector field corresponding to the hamiltonian function $h_t$. Consider the tubular neighborhood $U_{\delta}$ of $\partial W$. Let $\chi :U_{\delta}\to \mathbb{R}$ be such that $\chi=0$ on $U_{\delta/3}$ and $\chi=1$ on $U_{\delta}-U_{2\delta/3}$. \\

So consider the liouville form $\lambda'_t=\lambda+d(\chi h_t)$ with Liouville vector field $X+X_{\chi h_t}$. This new family Weinstein structures defined by $\lambda'_t$ has the desired property as near $\partial W$ the liouville vector field is $X$ which is the given liouville vector field and hence $\pitchfork$ to $\partial W$. We may need to adjust $\delta'$ in the definition of gradient like vector field in order to retain this property.

\end{proof}

 Now we shall show that on a tubular neighborhood of the stable manifold and away from the singularity $p_1\in L$, the fibers of $s$ are symplectic sub-manifolds, i.e $s$ induce a symplectic foliation.\\

So one needs to figure out $ker(ds)$. On the interior of a cube intersecting $L$, more precisely on the set $\beta=constant\neq 0$, we compute $ker(ds)$. 

\[
\bar{\partial}[(ae^{-Re(\Sigma z_j^2)/4}-b)+i(be^{-Re(\Sigma z_j^2)/4}+a)]\]
\[=\frac{e^{-Re(\Sigma z_j^2)/4}}{4}\Sigma_j[-(ax_j+by_j)dx_j+(-bx_j+ay_j)dy_j+i(-bx_j+ay_j)dx_j+i(ax_j+by_j)dy_j]
\]

\[
\partial[(ae^{-Re(\Sigma z_j^2)/4}-b)+i(be^{-Re(\Sigma z_j^2)/4}+a)]
\]
\[
=\frac{e^{-Re(\Sigma z_j^2)/4}}{4}\Sigma_j[(-ax_j+by_j)dx_j+(bx_j+ay_j)dy_j-i(bx_j+ay_j)dx_j+i(-ax_j+by_j)dy_j]
\]

\[
A[(ae^{-Re(\Sigma z_j^2)/4}-b)+i(be^{-Re(\Sigma z_j^2)/4}+a)]
\]
\[
=2\Sigma_j[(x_jdy_j-y_jdx_j)(be^{-Re(\Sigma z_j^2)/4}+a)-i(x_jdy_j-y_jdx_j)(ae^{-Re(\Sigma z_j^2)/4}-b)]
\]

Combining together we get 

\[
d_As=\Sigma_j[\{-\frac{e^{-Re(\Sigma z_j^2)/4}}{2}ax_j-2y_j(a+be^{-Re(\Sigma z_j^2)/4})\}dx_j
\]
\[
+\{\frac{e^{-Re(\Sigma z_j^2)/4}}{2}ay_j+2x_j(a+be^{-Re(\Sigma z_j^2)/4})\}dy_j]
\]
\[
+i\Sigma_j[\{-\frac{e^{-Re(\Sigma z_j^2)/4}}{2}bx_j+2y_j(-b+ae^{-Re(\Sigma z_j^2)/4})\}dx_j
\]
\[
+\{\frac{e^{-Re(\Sigma z_j^2)/4}}{2}by_j-2x_j(-b+ae^{-Re(\Sigma z_j^2)/4})\}dy_j]
\]

 $ker(d_As)$ is given by  the equation $BZ=0$, where $B$ is the matrix 

\[ \left( \begin{array}{ccccc}
 u_{11} & v_{11} & ... & u_{1n} & v_{1n}\\
 u_{21} & v_{21} & ... & u_{2n} & v_{2n}\\ 
  \end{array}\right)\]
 
where $*=2e^{Re(\Sigma z_j^2)/4}(be^{-Re(\Sigma z_j^2)/4}+a)$, $\tilde{*}=2e^{Re(\Sigma z_j^2)/4}(ae^{-Re(\Sigma z_j^2)/4}-b)$, $u_{1j}=-ax'_j-2y'_j*$, $v_{1j}=ay'_j+2x'_j*$, $u_{2j}=-bx'_j+2y'_j\tilde{*}$, $v_{2j}=by'_j-2x'_j\tilde{*}$ and $Z=(r_1\ s_1\ ...\ r_n\ s_n)^T$, $x'_j=\frac{e^{-Re(\Sigma z_j^2)/4}}{2}x_j$ and $y'_j=\frac{e^{-Re(\Sigma z_j^2)/4}}{2}y_j$. By making $r_2,s_2,...,r_n,s_n$ arbitrary we get the equation 

\[
\begin{array}{rcl}
-r_1(ax'_1+2y'_1*)+s_1(ay'_1+2x'_1*) &=& \bigstar_1\\
-r_1(bx'_1-2y'_1\tilde{*})+s_1(by'_1-2x'_1\tilde{*}) &=& \bigstar_2\\
\end{array}
\]
Where $\bigstar_1=\Sigma_{j=2}^n(u_{1j}r_j+v_{1,j}s_j)$ and $\bigstar_2=\Sigma_{j=2}^n(u_{2j}r_j+v_{2,j}s_j)$. The coefficient matrix of this system of equation is given by 

\[ E=\left( \begin{array}{cc}
 -(ax'_1+2y'_1*) & (ay'_1+2x'_1*) \\
 -(bx'_1-2y'_1\tilde{*}) & (by'_1-2x'_1\tilde{*}) \\ 
  \end{array}\right)\]
Whose determinant is given by $2({x'_1}^2-{y'_1}^2)(b*+a\tilde{*})$. So when $x'_1\neq \pm y'_1$ we can solve this system of equations and the solution space is of dimension $2n-2$, which means outside the cube containing the singularity $p_1$ we can solve this system of equation on the tubular neighborhood of the stable manifold $L'$, provided the tubular neighborhood is thin enough. \\

Now we compute $r_1$ and $s_1$ by assuming $n=4$, i.e, besides $r_1,s_1$ there are only $r_2$ and $s_2$, the general case is similar. We get \[r_1=\frac{1}{det(E)}[-abx_2y_1r_2-aby_1y_2s_2+bx_1x_2r_2-bx_1y_2s_2+M]\] and \[s_1=\frac{1}{det(E)}[a^2y_1x_2r_2+a^2y_1y_2s_2-abx_1x_2r_2-a^2x_1y_2s_2+N]\] where $M$ and $N$ are terms which contains $*$ and (or, both) $\tilde{*}$ as factors. Observe that $*$ and $\tilde{*}$ are very small if $a,b$ are small which is the case. \\

Observe that 

\[
\begin{array}{rcl}
\omega((r_1,s_1,r_2,s_2),(r'_1,s'_1,r'_2,s'_2)) &=& dx_1\wedge dy_1((r_1,s_1),(r'_1,s'_1))+dx_2\wedge dy_2((r_2,s_2),r'_2,s'_2)\\
 &=& (r_2s'_2-r'_2s_2)[1+\{(a^3b+2a^2b-a^2b^2)x_1x_2y_1y_2\\
 & & -(a^2b+ab^2)x_1^2x_2y_2\}+F(*,\tilde{*},a,b,x_j,y_j)]
\end{array}
\]

 So we see that if $k$ is large then $x$ is very small and $|y_j|<2$ (recall the translations) on the tubular neighborhood of the stable manifold $L'$ and hence $ker(d_As)$ restricted to this tubular neighborhood  is symplectic. As we have observed that on the tubular neighborhood of $L'$, both $r_1$ and $s_1$ are very small in comparison to $r_2$ and $s_2$ for high values of $k$, hence $ker(d_As)$ is small perturbation of the space generated by $r_2$ and $s_2$ which is the $x_2y_2$-space. So its symplectic complement $\perp_{\omega}$ is a small perturbation of the $x_1y_1$-plane but the morse function $\phi$ of the Weinstein structure in this coordinates is given by $\Sigma_jy_j^2$ and hence any vector field $X$ in the symplectic complement of $ker(ds)$ will have the property \[|X\phi|< \infty\] So parallel transport along a smooth path in the base $\mathbb{C}$ gives rise to a diffeomorphism between corresponding fibers as in Lemma 6.5 of \cite{Pardon}.\\

Define $W_0=(s\circ F)^{-1}(0)$ and the Weinstein structure is the pull-back Weinstein structure (See \ref{iso-sym pull}) by this iso-symplectic immersion $F$. Observe that $W_0$ is a symplectic submanifold because $F$ is an immersion between equi-dimensional manifolds and hence $\pitchfork$ to any submanifold in $W$ so it is also $\pitchfork$ to $s^{-1}(0)$. The fact that it is symplectic follows from the fact that $s^{-1}(0)$ is symplectic and $F$ is iso-symplectic.\\

 Let us now define the Lagrangian sphere required in \ref{Weinstein Lefschetz fibration}. Under parallel transport along the radial paths in $D^2\subset \mathbb{C}$ with respect to the symplectic connection induced by $\omega$, the critical points of $s$ sweeps out a Lagrangian disc called \textit{Lefschetz thimble} (\cite{Seidel}) which is the stable manifold of the Hamiltonian vector field $X_{Im \log s}$. The fiber over $0$ of a Lefschetz thimble is an exact Lagrangian sphere called the \textit{vanishing cycle}. Let $L'_1$ be the vanishing cycle corresponding to the singularity $p_1$ and let $F^{-1}(L'_1)=L_1\subset W_0$. Then the required Weinstein Lefschetz fibration is $((W_0,\lambda_0,\phi_0);L_1)$.\\

{\bf Acknowledgement:} We would like to thank Prof. Y. Eliashberg for his support and the reviewer from a previous submission for his comments. \\


\begin{thebibliography}{99}
\bibitem{Donaldson} Donaldson, S.K. Symplectic submanifolds and almost-complex geometry, J.Differential Geom. 44(1996), no. 4, 666-705. MR 1438190(98h:53045) 3,5,6,12,13,16

\bibitem{Datta} Datta, M. Islam, M.R. Submersions on open symplectic manifolds, Topol. Appl. 156(10) (2009) 1801-1806.

\bibitem{Eliashberg} Eliashberg, Yakov; Ganatra, Sheel; Lazarev, Oleg. Flexible Lagrangians. arXiv. 

\bibitem{Mishachev}  Eliashberg, Y.; Mishachev, N. Introduction to the h-principle. Graduate Studies in Mathematics, 48. American Mathematical Society, Providence, RI, 2002. xviii+206 pp. ISBN: 0-8218-3227-1 (Reviewer: John B. Etnyre)

\bibitem{Pardon} Giroux, Emmanuel; Pardon, John. Existence of Lefschetz fibrations on Stein and Weinstein domains. Geom. Topo. 21(2017), no.2,963--997.

\bibitem{Weinstein} Weinstein, Alan. Contact surgery and symplectic handlebodies, Hokkaido Math. J. 20(1991), n0.2, 241-251. MR 1114405 (92g:53028) 20, 24

\bibitem{Hormander} H\"{o}rmander, Lars. An introduction to Complex Analysis in Several Variables, North Holand, Amsterdam, 1973.

\bibitem{Kai} Cieliebak, Kai; Eliashberg, Yakov. \textit{From Stein to Weinstein and back}, American Mathematical Society Colloquium Publications, vol. 59, American Mathematical Society, Providence, RI, 2012, Symplectic geometry of affine complex manifolds. MR 3012475 1,4,19,20,24,25

\bibitem{Seidel} Seidel, Paul; Fukaya Categories and Picard-Lefschetz Theory, Volumn 10, Zurich lectures in advanced mathematics.
 
 \end{thebibliography}
\end{document}